\documentclass[11pt, reqno]{amsart}
\usepackage[style=alphabetic, sorting = nyt, maxnames=99]{biblatex}
\usepackage{geometry}
\usepackage{amssymb}
\usepackage{setspace}
\usepackage{bm}
\usepackage{mathtools}
\usepackage{float}
\usepackage{amsmath, euscript}
\usepackage{amssymb}
\usepackage{amsthm}
\usepackage{array}
\usepackage{mathrsfs}
\usepackage[all]{xy}
\usepackage{fancyhdr}
\usepackage{enumitem}
\usepackage{subcaption}
\usepackage{tabularx}
\usepackage[usenames, dvipsnames]{xcolor}
\usepackage{hyperref}
\hypersetup{colorlinks=true,linkcolor=purple,citecolor=purple}
\usepackage{tikz-cd}
\usepackage{svg}

\theoremstyle{plain}
\newtheorem{theorem}{Theorem}[section]
\newtheorem{thmy}{Theorem}
\newenvironment{thmx}{\stepcounter{theorem}\begin{thmy}}{\end{thmy}}

\newtheorem{lemma}[theorem]{Lemma}
\newtheorem{prop}[theorem]{Proposition}
\newtheorem{cor}[theorem]{Corollary}

\theoremstyle{definition}
\newtheorem{definition}[theorem]{Definition}
\newtheorem{example}[theorem]{Example}

\theoremstyle{remark}
\newtheorem*{remark}{Remark}

\renewcommand{\ss}{\mathfrak{s}}
\newcommand{\spinc}{Spin^{c}}

\numberwithin{equation}{section}

\DeclareMathOperator{\symg}{\textsf{Sym}^{g}(\Sigma_{g})}

\renewcommand{\ss}{\mathfrak{s}}

\newcommand{\aalpha}{\bm{\alpha}}
\newcommand{\bbeta}{\bm{\beta}}
\newcommand{\ddelta}{\bm{\delta}}
\newcommand{\ggama}{\bm{\gamma}}
\newcommand{\Ttheta}{\bm{\Theta}}
\newcommand{\ttau}{\bm{\tau}}
\newcommand{\zzeta}{\bm{\zeta}}
\newcommand{\xxi}{\bm{\xi}}

\newcommand{\ta}{\mathbb{T}_{\aalpha}}
\newcommand{\tb}{\mathbb{T}_{\bbeta}}

\newcommand{\xx}{\mathbf{x}}
\newcommand{\yy}{\mathbf{y}}

\newcommand{\cc}{\mathfrak{c}}
\newcommand{\Z}{\mathbb{Z}}

\newcommand{\F}{\mathbb{F}_{2}}

\newcommand{\R}{\mathbb{R}}

\newcommand{\D}{\mathfrak{D}}
\newcommand{\Dbar}{\underline{\mathfrak{D}}}
\newcommand{\Sigmabar}{\underline{\Sigma}}
\renewcommand{\aa}{\mathfrak{a}}
\newcommand{\aalphabar}{\underline{\aalpha}}
\newcommand{\bb}{\mathfrak{b}}
\newcommand{\bbetabar}{\underline{\bbeta}}

\newcommand{\ggamabar}{\underline{\ggama}}
\newcommand{\tg}{\mathbb{T}_{\ggama}}

\newcommand{\reduced}{\textsf{red}}

\newcommand{\disk}{\mathbb{D}^{2}}
\newcommand{\Xbar}{\underline{X}}

\renewcommand{\L}{\mathbb{L}}

\newcommand{\tabar}{\mathbb{T}_{\aalphabar}}
\newcommand{\tgbar}{\mathbb{T}_{\ggamabar}}
\newcommand{\tbbar}{\mathbb{T}_{\bbetabar}}
\renewcommand{\H}{\mathcal{H}}
\newcommand{\conn}{\EuScript{H}}

\newcommand{\crit}{\textsf{Crit}}

\newcommand{\comment}[1]{}

\DeclareMathOperator{\coker}{coker}

\definecolor{RED}{RGB}{200,0,0}
\definecolor{BLUE}{RGB}{0,0,200}
\definecolor{GREEN}{RGB}{0,100,0}

\usepackage[all]{nowidow}
\numberwithin{equation}{section}

\title{Trisections and Ozsv\'ath-Szab\'o Cobordism Invariants}
\author{William E. Olsen}
\date{\today}

\begin{document}
\maketitle

\begin{abstract}
Given a smooth, compact four-manifold $X$ viewed as a cobordism from the empty set to its connected boundary, we demonstrate how to use the data of a trisection map $\pi:X^{4} \to \mathbb{R}^{2}$ to compute the induced cobordism maps on Heegaard Floer homology associated to $X$.
\end{abstract}

\setcounter{tocdepth}{1}
\tableofcontents

\section{Introduction}
\label{sec:Introduction}

A new tool in smooth four-manifold topology has recently been introduced under the name of \emph{trisected Morse $2$-functions} (or \emph{trisections} for short) by Gay and Kirby \cite{gay2016trisecting}. Recent developments in this area demonstrate rich connections and applications to other aspects of four-manifold topology, including a new approach to studying symplectic manifolds and their embedded submanifolds \cite{lambert2020symplectic, lambert2019symplectic, lambert2018bridge}, and to surface knots (embedded in $S^{4}$ and other more general $4$-manifolds) \cite{meier2017bridge, meier2018bridge} along with associated surgery operations \cite{gay2018doubly, kim2020trisections}.

Of particular interest to the trisection community is the construction of new \cite{kirby2018new, castro2019relative} and the adaptation of established invariants in the trisection framework. In this article, we are concerned with the latter as we endeavor to demonstrate a technique for computing the Heegaard Floer cobordism maps from the data of a relative $(g,k;p,b)$-trisection map (see Sections \ref{sec:back-HF} and \ref{sec:Background-Trisections} for definitions). Our main result can be summarized as follows (see Section \ref{sec:computing-rel-invts} for more precise statements):

\begin{thmx}
\label{thm:Main}
  Fix a smooth, connected, oriented, compact four-manifold $X$ with connected boundary $\partial X = Y$, and let $\pi: X \to \mathbb{R}^{2}$ be a (relative) trisection map. Using $\pi$ as input data, one can recover the induced cobordism maps in Heegaard Floer homology
  \begin{equation}
      \label{eq:cob-map}
  F^{\circ}_{X, \ss}: HF^{\circ}(S^{3}) \to HF^{\circ}(Y, \ss|_{Y}),
  \end{equation}
  where $X$ is viewed as a cobordism from $S^{3}$ to $Y$ after removing the interior of a small ball, and $\circ \in \{+, -, \infty, \wedge\}$ are the variants defined in \cite{ozsvath2004holomorphic}.
\end{thmx}

Theorem \ref{thm:Main} has a few interesting characteristics and implications that may be worth mentioning. The first is that Ozsv\'ath and Szab\'o prove that their induced maps are smooth invariants of the underlying cobordism \cite[Theorem 1.1]{ozsvath2006holomorphic} which implies that Theorem \ref{thm:Main} may be used in the detection of exotic phenomena.  For example, following the usual Mayer-Vietoris strategy found in Floer homology theories, it is theoretically possible to use Theorem \ref{thm:Main} to recover the \emph{mixed invariants} \cite[Theorem 9.1]{ozsvath2006holomorphic} of a closed four-manifold $X$ with $b_{2}^{+}(X) > 1$. However, we warn the reader that, practically speaking, using Theorem \ref{thm:Main} to compute the mixed invariants by hand in any particular example remains a daunting task due to the general unruliness of pseudo-holomorphic curves.

The second feature we'd like to highlight, and perhaps most important in the author's opinion, is that Theorem \ref{thm:Main} makes no reference to a handle decomposition of the underlying four-manifold. Instead, the theorem takes as input data a (definite) broken fibration (which has been isotoped into a special form) and manages to return the Heegaard Floer cobordism maps as output. As such, Theorem \ref{thm:Main} may give some insight into how one might compare the relative invariants arising in different Floer homology theories--most notably a comparison between the Oszv\'ath-Szab\'o mixed invariants and Perutz's Lagrangian matching invariants \cite{perutz2007lagrangian, perutz2008lagrangian}.

With these preliminary remarks in place, we quickly summarize the proof of Theorem \ref{thm:Main}. In brief, a trisection map $\pi: X \to \mathbb{R}^{2}$ is a singular fibration over the disk whose singular set is of a prescribed type (assuming $X$ has non-empty boundary, the singular set has indefinite folds/cusps, none of which intersect the boundary). The central fiber (preimage of $(0,0)$ under $\pi$) is a genus $g$ surface with $b > 0$ boundary components and is decorated with three sets of `vanishing cycles', colored red, blue, and green. A fundamental result of Gay and Kirby \cite{gay2016trisecting} is that $\pi$ induces an open book decomposition of the boundary $3$-manifold $Y = \partial X$ whose monodromy can be recovered by flowing a regular fiber once around the boundary of the disk (after choosing the appropriate auxiliary data, such as a metric and compatible connection). By starting at the central fiber, flowing to the boundary, once around, and then back to the center\footnote{Although we don't directly study holomorphic sections of $\pi$, the author was deeply inspired by the constructions of \emph{Lagrangian boundary conditions} found in \cite{seidel2008fukaya} and \cite{perutz2007lagrangian}. The approach taken here is meant to be reminiscent of these constructions.}, one obtains a Heegaard triple which is slide-equivalent to one which is subordinate to a bouquet for a framed link as in \cite{ozsvath2006holomorphic}--diagrammatically, the result is a \emph{closed} surface decorated by three complete sets of attaching curves and a canonical choice of basepoint. Theorem \ref{thm:Main} then follows via the usual naturality considerations. 

\subsection*{Organization}
\label{subsec:organization}

This note is organized as follows. In Sections \ref{sec:Background-Trisections} and \ref{sec:back-HF}, we briefly review the necessary background behind trisected Morse-$2$ functions and the induced cobordism maps in Heegaard Floer homology. The heart of the paper is found in Section \ref{sec:computing-rel-invts} where we give the details of the construction outlined above for how to use the data of a relative trisection map to compute the Ozsv\'ath-Szab\'o cobordism maps. In the last section we comment on the role of the contact class \cite{ozsvath2005heegaard, honda2009contact} in our setup.

\subsection*{Acknowledgements}
This work would not have been possible without the insight and generous support of my Ph.D. advisor, David Gay. Also I'd like to thank Juanita Pinz\'on-Caicedo and John Baldwin for their interest in this project. Finally, I would like to thank the Max Planck Institute for Mathematics in Bonn, Germany, for hostimg me while I worked towards completing this project.

\newpage

    \section{Trisections of four-manifolds}
    \label{sec:Background-Trisections}

The literature is rich with helpful and insightful constructions of the trisection theory. For this reason, we only briefly review its foundational material and point the interested reader elsewhere for a less terse introduction. For a general overview of trisections and direct comparisons with the more familiar description of four-manifolds via handle decompositions and Kirby calculus, we recommend the original \cite{gay2016trisecting} and the more recent survey \cite{gay2019heegaard}. For interesting examples of trisections and their diagrams, including descriptions for various surgery operations such as the Gluck twist and its variants, we recommend \cite{kim2020trisections, lambert2018bridge, gay2018doubly, aranda2019diagrams} and \cite{koenig2017trisections}. For a broader perspective on stable maps from four-manifolds to surfaces, including details about how to simplify the topology of such maps, we suggest \cite{gay2015indefinite, gay2012reconstructing, baykur2017simplifying} and the references therein.

\subsection{Heegaard splittings and the essentials of diagrammatic representations of manifolds}
\label{subsec:relative-trisection-diagrams}

In this first section, we establish a vocabulary and notation for discussing the diagrammatic representations of $3$- and $4$-manifolds which are prevalent throughout. When possible, we closely follow the terminology and notation of \cite[Section 2.2]{gay2018doubly} and \cite[Section 2]{juhasz2012naturality}.

We start with the essentials:
\begin{itemize}
    \item $\Sigma_{g, b}$ is a compact, oriented, connected surface of genus $g$ with $b$ boundary components.
    \item $\ddelta = \{\delta_{1}, \ldots, \delta_{g-p}\} \subset \Sigma_{g,b}$ is a \emph{genus $p$ cut system}, i.e. a collection of disjoint simple closed curves which collectively $\Sigma_{g,b}$ into a connected genus $p$ surface. In symbols, $\Sigma_{g,b} \setminus \cup_{i} \delta_{i} \cong \Sigma_{p,b}$.
    \item Two genus $p$ cut systems $\ddelta$ and $\ddelta'$ are said to be \emph{slide-equivalent} if there exists a sequence of handle-slides taking $\ddelta$ to $\ddelta'$ (see \cite[Section 2]{juhasz2012naturality} for a picture of a handle-slide). Moreover, $\ddelta$ and $\ddelta'$ are said to be \emph{slide-diffeomorphic} if there exists a diffeomorphism $\phi: \Sigma \to \Sigma$ such that $\phi(\ddelta)$ is slide-equivalant to $\ddelta'$.
\end{itemize}

An important fact that we'll use repeatedly is the following\footnote{Strictly speaking, we should be more careful about distinguishing genus $p$ cut systems from \emph{isotopy classes} of such--this point of view is taken in \cite{juhasz2012naturality}. The interested reader is also pointed there for more details about (sutured) compression bodies.}: a genus $p$ cut system on $\Sigma_{g,b}$ determines (up to diffeomorphism rel. boundary) a \emph{compression body} $C_{\ddelta}$ which is the cobordism obtained from $I \times \Sigma_{g,b}$ by attaching three-dimensional $2$-handles to $\{1\} \times \Sigma_{g,b}$ along the curves $\{1\} \times \ddelta$. Moreover, any two such compression bodies $C_{\ddelta}$, $C_{\ddelta'}$ are diffeomorphic rel. boundary if and only if $\ddelta, \ddelta'$ are slide-diffeomorphic.

\begin{definition}
Let $Y$ be a connected, oriented $3$-manifold which may or may not have boundary. An \emph{embedded genus-$g$ Heegaard diagram} for $Y$ is a triple $(\Sigma_{g,b}, \aalpha, \bbeta)$ where $\aalpha, \bbeta$ are genus $p$ cut systems on $\Sigma_{g,b}$ which respectively bound compressing disks on either side of $\Sigma_{g,b}$. If $\partial Y = \emptyset$, then $p = g$.
\end{definition}

The relevance of the above definitions is that if $(\Sigma_{g,b}, \aalpha', \bbeta')$ is another abstract\footnote{There is a great deal of subtlety when comparing `abstract' diagrams, by which we mean a picture of a surface decorated with colored collections of curves, and `embedded' diagrams, as we've defined above. For more on this subtlety and how its relevant to the Heegaard Floer computatinos which come later, we recommend \cite{juhasz2012naturality} for an excellent discussion and examples.} diagram with $\aalpha \sim \aalpha'$ and $\bbeta \sim \bbeta'$ slide-equivalent, then $(\Sigma_{g,b}, \aalpha', \bbeta')$ is also a Heegaard diagram for $Y$. Every oriented, connected $3$-manifold $Y$ admits an embedded Heegaard diagram.

In fact, up to diffeomorphism rel. $\Sigma_{g,b}$, an abstract diagram $(\Sigma_{g,b}, \aalpha, \bbeta)$ determines a $3$-manifold $Y$ via the following procedure: start with $[-1,1] \times \Sigma_{g,b}$ and attach $2$-handles to $\{-1\} \times \Sigma_{g,b}$ along the curves $\{-1\} \times \aalpha$, and similarly attach $2$-handles to $\{1\} \times \Sigma_{g,b}$ along $\{1\} \times \bbeta$. For more details about the precise relationship between abstract diagrams, embedded diagrams, and their associated classes of $3$-manifolds, we refer the reader to \cite[Section 2]{juhasz2012naturality}.

\begin{definition}
\label{def:Heegaard-Triple}
A \emph{Heegaard triple} is a $4$-tuple $(\Sigma_{g}, \aalpha, \bbeta, \ggama)$ where $\Sigma_{g}$ is a closed, oriented surface and each of $\aalpha, \bbeta$, and $\ggama$ are genus $g$ cut systems on $\Sigma_{g}$.
\end{definition}

Just as (embedded) Heegaard diagrams determine a smooth $3$-manifold up to diffeomorphism rel. $\Sigma$, a Heegaard triple determines (up to diffeomorphism) a smooth $4$-manifold. To this end, let $H = (\Sigma, \aalpha, \bbeta, \ggama)$ be a Heegaard triple. In \cite[Section 8]{ozsvath2004holomorphic}, Ozsv\'ath and Szab\'o associate to $H$ a four-manifold $X_{\aalpha, \bbeta, \ggama}$ via

\begin{equation}
    \label{eq:Heegaard-Triple-4-manifold}
    X_{\aalpha, \bbeta, \ggama} : = \Big((\Sigma \times \Delta) \cup (U_{\aalpha} \times e_{\aalpha}) \cup (U_{\bbeta} \times e_{\bbeta}) \cup (U_{\ggama} \times e_{\ggama})\Big)/ \sim 
\end{equation}
where $\Delta$ is a triangle with edges labeled $e_{\aalpha}, e_{\bbeta}$, and $e_{\ggama}$ clockwise, and $\sim$ is the relation determined by gluing $U_{\ttau} \times e_{\ttau}$ to $\Sigma \times \Delta$ along $\Sigma \times e_{\ttau}$ for each $\ttau \in \{\aalpha, \bbeta, \ggama\}$ using the natural identification.

We note that if $H = (\Sigma, \aalpha, \bbeta, \ggama)$ is a general Heegaard triple, with no conditions on the pairwise cut systems, then the four-manifold $X_{\aalpha, \bbeta, \ggama}$ constructed in equation \eqref{eq:Heegaard-Triple-4-manifold} has three boundary components
\begin{equation}
    \label{eq:boundary-Xabg}
    \partial X_{\aalpha, \bbeta, \ggama} = -Y_{\aalpha, \bbeta} \sqcup - Y_{\bbeta, \ggama} \sqcup Y_{\aalpha, \ggama} 
\end{equation}
given by the three Heegaard splittings $(\Sigma, \aalpha, \bbeta)$, $(\Sigma, \bbeta, \ggama)$, and $(\Sigma, \ggama, \aalpha)$. 

However, if $H = (\Sigma, \aalpha, \bbeta, \ggama)$ is required to be a trisection diagram, so that we have
\[
(\Sigma, \aalpha, \bbeta) \cong (\Sigma, \bbeta, \ggama) \cong (\Sigma, \ggama, \aalpha) \cong \#^{k}S^{1} \times S^{2},
\]
then it follows (again from Laudenbach-Poenaru \cite{laudenbach1972note}) that we can fill in these three boundary components and obtain a closed four-manifold.

\subsection{Trisections as singular fibrations}
\label{subsec:trisections-closed-four-manifolds}

The theory of trisections arose from the study of generic smooth maps from four-manifolds to surfaces \cite{gay2012reconstructing, gay2015indefinite, gay2016trisecting}, and while the diagrammatic consequences of the trisection theory are certainly interesting, we'll maintain the historical perspective and view stable maps $\pi: X^{4} \to \disk$ as the primary object of interest. Along the way we explain how the familiar notions of connection, parallel transport, and vanishing cycles can be imported into this setting. To be clear, none of what's presented in this section is original, our main sources being the excellent work \cite{hayano2014modification, behrens2012vanishing, behrens2016elimination, behrens2014smooth}.

Fix $X$ to be a compact, oriented, connected smooth $4$-manifold with corners, and let $\disk$ be the unit disk in the plane. A \emph{stable} map $\pi: X \to \disk$ is one whose critical locus and critical image admit local coordinate descriptions of the following two types\footnote{We recommend \cite{hayano2014modification} for more details on stable maps.}:
\begin{enumerate}
    \item \label{itm:first} \emph{Indefinite fold model}: in local coordinates, $\pi$ is equivalent to:
    \begin{equation}
        \label{eq:indefinite-fold}
        (t,x,y,z) \mapsto (t, x^{2} + y^{2} - z^{2}) 
    \end{equation}
    
    \item \emph{Indefinite cusp model}: in local coordinates, $\pi$ is equivalent to
    \begin{equation}
        \label{eq:indefinite-cusp}
            (t,x,y,z) \mapsto (t, x^{3} + 3tx + y^{2} - z^{2})
            \end{equation}
\end{enumerate}

Since we're interested in four-manifolds with non-empty boundary, we impose the following additional constraints on $\pi$ near the boundary\footnote{See \cite[Remark 4.5]{baykur2017simplifying} for comments on how to extend the homotopy techniques of indefinite fibrations to manifolds with boundary.}:

\begin{enumerate}
\setcounter{enumi}{2}
    \item The boundary of $X$ decomposes into two codimension zero pieces, the \emph{horizontal part} $\partial^{h}X$ and the \emph{vertical part} $\partial^{v}X$, so that $\partial X = \partial^{v}X \cup \partial^{h}X$. The vertical part of the boundary is defined to be $\partial^{v}X = \pi^{-1}(\partial \disk)$, and the horizontal part is defined to be the closure of its complement. 
    \end{enumerate}
\noindent
An important aspect of the behavior of $\pi$ near the horizontal part of the boundary is the following regularity condition.
    \begin{enumerate}
    \setcounter{enumi}{3}
    \item The stable map $\pi$ restricts to smooth fibration on $\partial^{h}X$.
    \end{enumerate}
    \noindent
    The tangent space at any point $p \in X$ splits into
    \begin{equation}
        \label{eq:tangent-space}
        TX_{p} = T_{p}X^{h} \oplus T_{p}X^{v}
    \end{equation}
    where the vertical tangent space is defined to be $T_{p}X^{v} = \ker (D \pi)$, and the horizontal tangent space $T_{p}X^{h}$ is its orthogonal complement with respect to a chosen metric on $X$ (compare \cite{hayano2014modification}). 

\begin{enumerate}
    \setcounter{enumi}{4}
        \item \label{itm:last} If $x$ lies in $\partial^{h}X$, then the horizontal part of the tangent space lies in $T_{x}\partial^{h}X$.
\end{enumerate}

Our main concern in introducing the assumptions \eqref{itm:first} - \eqref{itm:last} is so that we can import the technology of \cite{hayano2014modification, behrens2014smooth, behrens2016elimination} into our setting where $X$ has boundary--namely, so that we can use the tools of $\pi$-compatible connections and parallel transport. For example, the splitting \eqref{eq:tangent-space} defines a \emph{$\pi$-compatible connection}, as in \cite{behrens2016elimination}. Moreover, since $T^{h}X$ is parallel to $\partial^{h}X$, these $\pi$-compatible connections have well-defined parallel transport maps (which are only partially defined on the fibers) in our case when $X$ has non-empty boundary. For more details about connections and parallel transport maps in the context of singular fibrations, we recommend \cite{behrens2016elimination}.

\begin{definition}
\label{def:trisection-map}
Let $X$ be a compact, oriented, connected smooth $4$-manifold with connected boundary $Y = \partial X$. A \emph{$(g,k; p,b)$-trisection map} $\pi: X \to \disk$ is a stable map satisfying the boundary conditions \eqref{itm:first} -- \eqref{itm:last} above and whose critical image is shown in Figure \ref{fig:tris-map} below.
\end{definition}

\begin{figure}[H]
    \centering
    \includegraphics[scale = 0.4]{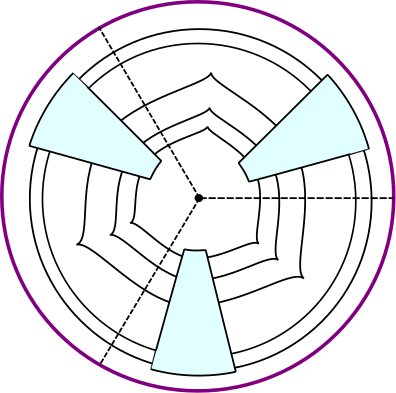}
    \caption{The image of the critical value set in a $(g,k; p,b)$-trisection map. The central fiber $\Sigma := \pi^{-1}(0,0)$ is an oriented genus $g$ surface with $b$ boundary components. The fibers over the purple boundary of the unit disk are the pages of an open book decomposition of the boundary $3$-manifold $Y$--the page has genus $p$. Each of the blue boxes denotes a \emph{Cerf box}, as in \cite[Figure 17]{gay2016trisecting}.}
    \label{fig:tris-map}
\end{figure}

From a $(g, k; p,b)$–trisection map $\pi:X \to \disk$ one can recover the standard decomposition of $X$ into three pieces as described in \cite{castro2018trisections}. Clearly, the three dotted line segments in Figure \ref{fig:tris-map} decompose the image of $\pi$ into three sectors $D_{1}$, $D_{2}$, and $D_{3}$. Define $Z_{i} : = \pi^{-1}(D_{i})$, and note that the local models described in equations \eqref{eq:indefinite-fold}--\eqref{eq:indefinite-cusp} imply that
\[
X = Z_{1} \cup Z_{2} \cup Z_{3}
\]
is naturally a $(g, k; p,b)$–trisection of $X$--see \cite{castro2018diagrams} for more details.

In \cite{gay2016trisecting}, Gay and Kirby show that such a stable map on a four-manifold induces an open book decomposition of its boundary $3$-manifold.

\begin{theorem}[\cite{gay2016trisecting}]
\label{thm:trisection-induces-open-book}
A relative $(g,k;p,b)$-trisection map $\pi: X \to \disk$ induces an open book decomposition on the boundary three-manifold.
\end{theorem}

Importantly, the monodromy diffeomorphism of the open book decomposition on the boundary $3$-manifold can be recovered combinatorially from the a trisection diagram $(\Sigma, \aalpha, \bbeta, \ggama)$ \cite[Theorem 5]{castro2018diagrams}. This is discussed in more detail in Section \ref{subsec:Heeg-Triples}.

\section{Background on Heegaard Floer homology}
\label{sec:back-HF}

This sections provides a brief review of those aspects of Heegaard Floer homology that will be most important to us: the Heegaard Floer chain complexes, the chain maps induced by four-dimensional cobordisms, and the definition of the `mixed' invariants for closed four-manifolds $X$ with $b_{2}^{+}(X) > 1$. We assume the reader is familiar with the Heegaard Floer canon \cite{ozsvath2004holomorphic, ozsvath2006holomorphic, lipshitz2006cylindrical}.

 \subsection{Heegaard Floer chain complexes}
\label{subsec:HF-chain-complexes}

Fix a closed, connected, oriented three-manifold $Y$, and denote by $\spinc(Y)$ the space of $\spinc$ structures on $Y$. Given a pointed Heegaard splitting $H = (\Sigma, \aalpha, \bbeta, w)$ of $Y$, Ozsv\'ath and Szab\'o \cite{ozsvath2004holomorphic, ozsvathPropsApps} study the Lagrangian Floer cohomology of the two tori
\[
\ta = \alpha_{1} \times \cdots \times \alpha_{g} \quad \tb = \beta_{1} \times \cdots \times \beta_{g}
\]
inside the symmetric product $\symg$. To review their construction, fix $\ss \in \spinc(Y)$ and recall the map $s_{w}: \ta \cap \tb \to \spinc(Y)$. Assuming the Heegaard splitting is \emph{admissable} for the $\spinc$-structure $\ss$ (see \cite{ozsvath2004holomorphic} for more details), and after choosing a (generic) family $J$ of almost complex structures on $\symg$, define the chain complex $CF^{\infty}(\H,\ss)$ to be freely generated over $\F$ by pairs $[x,i]$ where $\xx \in \ta \cap \tb$ satisfies $s_{w}(\xx) = \ss$, $i$ is an integer, $\F$ is the field of two elements, and $\H$ denotes the pair $\H = (H, J)$.

The differential on $CF^{\infty}(\H, \ss)$ is given by
\begin{equation}
    \label{eq:HF-differential}
    \partial [\xx, i] = \sum_{\substack{y \in \ta \cap \tb \\ s_{w}(\yy) = \ss}} \sum_{\substack{\phi \in \pi_{2}(\xx, \yy) \\ \mu(\phi) = 1}} \#\mathcal{M}(\phi) \cdot [\yy, i - n_{w}(\phi)],  
\end{equation}
where $\pi_{2}(\xx, \yy)$ is the space of homotopy classes of Whitney disks connecting $\xx$ to $\yy$, $\mu(\phi)$ is the Maslov index, $\mathcal{M}(\phi)$ is the moduli space of $J$-holomorphic disks in the class $\phi$ (modulo the action of $\R$), and $n_{w}(\phi)$ is the algebraic intersection number of $\phi$ with the divisor $\{w\} \times \mathsf{Sym}^{g-1}(\Sigma)$. The chain groups $CF^{\infty}(\H, \ss)$ come equipped with an $\F[U, U^{-1}]$-action, where $U$ acts by  $U \cdot [\xx, i] = [\xx, i -1]$ and is viewed as a formal variable of degree $-2$. With the infinity complex $(CF^{\infty}(\H, \ss), \partial)$ in hand, one obtains other complexes $CF^{+}, CF^{-}$, and $\widehat{CF}$ by restricting attention to pairs $[x, i]$ with $i \geq 0, i < 0$, and $i = 0$, respectively. The subsequent complexes have an induced $\F[U]$-action, which is trivial in the case of $\widehat{CF}$.

Clearly, the plus, minus, and infinity variations are related by a short exact sequence
\begin{equation}
    \label{eq:ses}
    0 \to CF^{-}(\H, \ss) \to CF^{\infty}(\H, \ss) \to CF^{+}(\H, \ss) \to 0
\end{equation}
which produces a long exact sequence in homology
\begin{equation}
    \label{eq:HF-LES}
    \cdots \to HF^{+}_{\ast+1}(Y, \ss) \xrightarrow{\delta} HF^{-}_{\ast}(Y, \ss) \to HF^{\infty}_{\ast}(Y, \ss) \to \cdots
\end{equation}
where $\delta: HF^{1}_{\ast+1}(Y, \ss) \to HF^{-}_{\ast}(Y, \ss)$ is a connecting homomorphism. Last, the \emph{reduced Heegaard Floer homology groups} $HF^{\pm}_{\reduced}(Y, \ss)$ are defined as
\begin{equation}
    \label{eq:HF-reduced-minus}
    HF^{-}_{\reduced}(Y, \ss) : = \ker \big(\iota_{\ast}: HF^{-}(Y, \ss) \to HF^{\infty}(Y, \ss)\big)
\end{equation}
and
\begin{equation}
    \label{eq:HF-reduced-plus}
    HF^{+}_{\reduced}(Y, \ss) : = \coker \big( \pi_{\ast}: HF^{\infty}(Y, \ss) \to HF^{+}(Y, \ss)\big).
\end{equation}
The connecting homomorphism $\delta$ induces an isomorphism from $HF^{+}_{\reduced}(Y, \ss)$ to $HF^{-}_{\reduced}(Y, \ss)$. Unlike $HF^{\pm}$, the modules $HF^{\pm}_{\reduced}(Y, \ss)$ are always finite-dimensional over $\F$.

 \subsection{Maps associated to cobordisms}
\label{subsec:2-handle-maps}

In addition to defining the $\F[U]$-modules $HF^{\circ}(Y, \ss)$, Ozsv\'ath and Szab\'o show that four-dimensional cobordisms between $3$-manifolds induce $\F[U]$-equivariant maps between the respective Floer homology groups. We briefly highlight the main points of this construction. Our exposition closely follows that of \cite[Section 4]{lipshitz2008combinatorial}.

Consider a smooth, connected, oriented four-dimensional cobordism $W$ from $Y_{-}$ to $Y_{+}$, where $Y_{\pm}$ are also closed, connected, and oriented, and fix a $\spinc$ structure $\ss$ on $W$. Let $f$ be a self-indexing Morse function on $W$, and consider the associated handle decomposition of $W$:
\[
Y_{-} \xrightarrow{W_{1}} Y_{1} \xrightarrow{W_{2}} Y_{2} \xrightarrow{W_{3}} Y_{+}
\]
where the cobordism $W_{i}$ contains only index-$i$ handles.

Given the data $(W = \bigcup_{i}W_{i}, f)$, Ozsv\'ath and Szab\'o \cite{ozsvath2006holomorphic} associate to $(W, \ss)$ an induced map $F^{\circ}_{W, \ss}: HF^{\circ}(Y_{-},\ss_{Y_{-}}) \to HF^{\circ}(Y_{+}, \ss_{Y_{+}})$ between the Floer homologies of $Y_{-}$ and $Y_{+}$ by first defining maps $F^{\circ}_{W_{i}, \ss|_{W_{i}}}$ for each $i = 1,2,3$ and then taking $F^{\circ}_{W, \ss}$ to be their composition. We now review the definitions of these three maps.

\subsubsection*{One- and three-handle maps}
\label{subsubsec:1-and-3-handle-maps}

Suppose that $W_{1}$ is a cobordism from $Y_{-}$ to $Y_{1}$ which consists entirely of $1$-handle additions, and let $\ss$ be a $\spinc$-structure on $W_{1}$. Since $W_{1}$ consists only of $1$-handles, it follows that $Y_{1} \cong Y_{-} \# (S^{1} \times S^{2})^{\# n}$ where $n$ is the number of $1$-handles added by $W_{1}$. By \cite{ozsvath2006holomorphic}, there is a non-canonical identification of the Floer homology of $Y_{1}$ and the module
\begin{equation}
    \label{eq:HF-splits}
  HF^{\circ}(Y_{1}, \ss_{Y_{-}} \# \ss_{0}) \cong HF^{\circ}(Y_{-}, \ss_{Y_{-}}) \otimes H_{\ast}(T^{n}; \Z)
\end{equation}
\noindent
where $H_{\ast}(T^{n};\Z)$ is the usual singular homology of the $n$-torus $T^{n}$ with integer coefficients. Let $\Theta^{+}$ denote the generator (we're working over $\F$) of the top-graded part of $H_{\ast}(T^{n}, \Z)$. The cobordism map $F^{\circ}_{W_{1}, \ss}$ is defined on generators to be
\begin{equation}
    \label{eq:1-handle-map}
F^{\circ}_{W_{1}, \ss}: HF^{\circ}(Y_{0}, \ss_{Y_{0}}) \to HF^{\circ}(Y_{1}, \ss_{Y_{1}}) \qquad [\xx, i] \mapsto [\xx \otimes \Theta^{+}, i]
\end{equation}

It is proved in \cite[Section 4.3]{ozsvath2006holomorphic} that, up to composition with canonical isomorphisms, $F^{\circ}_{W_{1}, \ss}$ does not depend on the choices made in its construction. For brevity, we will usually denote the $1$-handle map by $F_{1}$.

Next, if $W_{3}$ is a cobordism which can be built using only $3$-handles, then for $\ss \in \spinc(W_{3})$ the map $F^{\circ}_{W_{3}, \ss_{3}}$ is defined as the dual of $F^{\circ}_{W_{1}, \ss_{1}}$ by
\[
F^{\circ}_{W_{3}, \ss_{3}}: HF^{\circ}(Y_{2}, \ss_{Y_{2}}) \to HF^{\circ}(Y_{+}, \ss_{Y_{+}}) \qquad [\xx \otimes \Theta^{-}, i] \mapsto [\xx, i]
\]
where $\Theta_{-}$ is the generator of the lowest-graded part of $H_{\ast}(T^{n}, \Z)$. Moreover, 
\[
F^{\circ}_{W_{3}, \ss_{3}}([\xx \otimes \xi, i]) = 0
\]
for any homogeneous generator $\xi$ which does not lie in the bottom degree. Again, the map is independent of the choices made in its construction (e.g. choice of splitting \eqref{eq:HF-splits}).

\subsubsection*{Two-handle maps}
\label{subsubsec:2-handle-maps}

Suppose now that $W$ consists only of $2$-handle additions. In \cite[Definition 4.2]{ozsvath2006holomorphic} Ozsv\'ath and Szab\'o describe such cobordisms with a special type of Heegaard triple diagrams, as we now describe. Since $W$ from $Y_{1}$ to $Y_{2}$ consists only of $2$-handle additions, the cobordism corresponds to surgery on some framed link $\L \subset Y_{1}$. Denote by $\ell$ the number of components of $\L$, and fix a basepoint in $Y_{1}$. Let $B(\L)$ be the union of $\L$ with a path from each component to the basepoint. The boundary of a regular neighborhood of $B(\L)$ is a genus $\ell$ surface, which has a subset identified with $\ell$ punctured tori $F_{i}$, one for each link component. 

\begin{definition}
\label{def:Heeg-triple-sub-to-bouquet}
A Heegaard triple $(\Sigma, \aalpha, \bbeta, \ggama, w)$ is said to be \emph{subordinate to a bouqet} $B(\L)$ for the framed link $\L$ if
\begin{enumerate}[label={(\bfseries B\arabic*)}]
    \item \label{itm:bouq-1}$(\Sigma, \{\alpha_{1}, \ldots, \alpha_{g}\}$, $\{\beta_{1}, \ldots, \beta_{g-\ell}\})$ describes the complement of $B(\L)$.
    
    \item \label{itm:bouq-2}$\{\gamma_{1}, \ldots, \gamma_{g-\ell}\}$, are small isotopic translates of $\{\beta_{1}, \ldots, \beta_{g-\ell}\}$
    
    \item \label{itm:bouq-3}After surgering out the $\{\beta_{1}, \ldots, \beta_{g-\ell}\}$, the induced curves $\beta_{i}$ and $\gamma_{i}$, for $i = g-\ell+1, \ldots, g$, lie on the punctured torus $F_{i}$.
    
    \item \label{itm:bouq-4}For $i = g-\ell+1, \ldots, g$, the curves $\beta_{i}$ represent meridians for the link components, disjoint from all $\gamma_{j}$ for $i \neq j$, and meeting $\gamma_{i}$ in a single transverse point.
    
    \item \label{itm:bouq-5}for $i = g-\ell + 1, \ldots, g$, the homology classes of the $\gamma_{i}$ correspond to the framings of the link components.
\end{enumerate}
\end{definition}

The following lemma shows that one can represent the cobordism $W(\L)$ via a Heegaard triple subordinate to a bouquet for the framed link $\L$. For a proof, see for example \cite[Lemma 9.4]{zemke2015graph} or \cite[Proposition 4.3]{ozsvath2006holomorphic}.

\begin{lemma}
\label{lem:sub-bouq-2-handle-attachment}
Suppose $(\Sigma, \alpha, \bbeta, \ggama, w)$ is a Heegaard triple that is subordinate to a bouquet for a framed link $\L$ in $Y$. After filling in the boundary component $Y_{\bbeta, \ggama}$ with $3$- and $4$-handles, we obtain the handle cobordism $W(Y, \L)$.
\end{lemma}

We now define the cobordism maps for $2$-handle cobordisms. Suppose $\L \subset Y$ is a framed link in $Y$, and $B(\L)$ is a bouquet. Let $(\Sigma, \aalpha, \bbeta, \ggama,w)$ be a Heegaard triple subordinate to $B(\L)$. Let $\Theta \in \tb \cap \tg$ denote the intersection point in top Maslov grading \cite[Section 2.4]{ozsvath2006holomorphic}.

If $\ss \in \spinc(W(Y, \L))$, the $2$-handle map
\[
F^{-}_{\L, \ss}:CF^{-}(\Sigma, \aalpha, \bbeta, w, \ss|_{Y}) \to CF^{-}(\Sigma, \bbeta, \ggama, w, \ss|_{Y(\L)})
\]
is defined as a count of holomorphic triangles
\begin{equation}
    \label{eq:hol-triangles}
    F^{-}_{\L, \ss}([\xx,i]):= \sum_{\yy \in \ta \cap \tg} \sum_{\substack{\psi \in \pi_{2}(\xx, \Theta_{\beta, \ggama, \yy} \\ \mu(\psi) = 0 \\ s_{w}(\psi) = \ss}} \#\mathcal{M}(\psi) \cdot [i - n_{w}(\psi)],
\end{equation}
where $\pi_{2}(\xx, \Theta, \yy)$ is the set of homotopy classes of Whitney triangles with vertices $\xx, \Theta, \yy$, and $\mathcal{M}(\varphi)$ is the moduli space of holomorphic representatives of $\varphi$. 

Throughout Section \ref{sec:computing-rel-invts}, we will be interested in studying the holomorphic triangle map \eqref{eq:hol-triangles} for diagrams which are not a priori subordinate to a bouquet for a framed link. We address this issue there.

\section{Trisections and Ozsv\'ath-Szab\'o cobordism invariants}
\label{sec:computing-rel-invts}

In this section we prove Theorem \ref{thm:Main} and demonstrate how one can use the data of a relative trisection map $\pi: X^{4} \to \R^{2}$, along with some auxiliary data, to compute the induced cobordism maps in Heegaard Floer homology.

\subsection{Constructing Heegaard triples from relative trisection diagrams}
\label{subsec:Heeg-Triples}

Fix $X^{4}$ to be a compact, oriented, connected, smooth $4$-manifold. The input data we require is a tuple $(\pi,\langle \cdot, \cdot \rangle, \conn)$ consisting of a $(g,k;p,b)$-trisection map $\pi: X \to \R^{2}$, a Riemannian metric $\langle \cdot, \cdot \rangle$ on $X$, and a $\pi$-compatible connection $\conn$. Equipped with such data, we may choose three reference arcs $\eta_{\aalpha}, \eta_{\bbeta}, \eta_{\ggama}: [0,1] \to \disk$ as in Figure \ref{fig:three-Morse-functions} below.
\begin{figure}[ht]
    \centering
    \includegraphics[scale = 0.4]{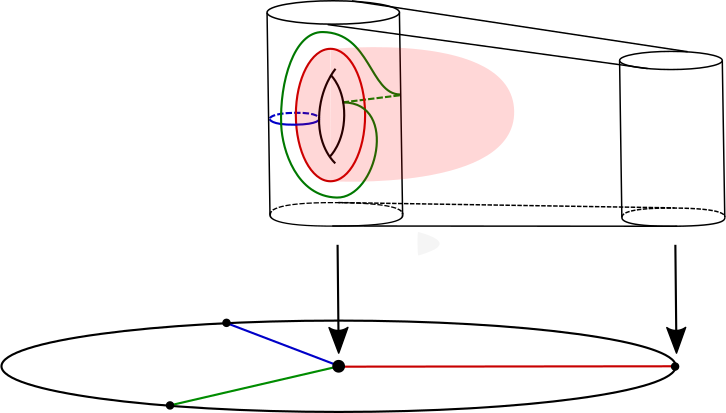}
    \caption{For $\ttau \in \{\aalpha, \bbeta, \ggama\}$, we have reference arcs $\eta_{\ttau}: [0,1] \to \disk$ for which $f_{\ttau}: U_{\ttau} \to [0,3]$ is a Morse function. Drawn in pink is the descending manifold for an index $2$ critical point for $f_{\aalpha}$ whose intersection with $\Sigma$ is an $\aalpha$ curve drawn in red.}
    \label{fig:three-Morse-functions}
\end{figure}
As discussed in Section \ref{subsec:trisections-closed-four-manifolds}, associated to these reference arcs are three Morse functions $f_{\aalpha}, f_{\bbeta}, f_{\ggama}$ defined on the compression bodies $U_{\aalpha}, U_{\bbeta}$, and $U_{\ggama}$ respectively. For $\ttau \in \{\aalpha, \bbeta, \ggama\}$, these Morse functions satisfy:
\begin{itemize}
    \item $f_{\ttau}:U_{\ttau} \to [0, 3]$ is a Morse function with $f_{\ttau}^{-1}(0) = \Sigma$ and $f_{\ttau}^{-1}(3) = \Sigma_{\ttau}$ the surface obtained by doing surgery on $\Sigma$ along the $\ttau$-curves; and,
    \item $f_{\ttau}$ has $g-p$ index two critical points whose descending manifolds intersect $\Sigma$ along the $\ttau$ curves.
\end{itemize}

We define the surface $\Sigma_{\aalpha}$ to be the fiber $f_{\aalpha}^{-1}(3)$ and we fix an identification of $\Sigma_{\aalpha}$ with $\Sigma_{p,b}$. Next, endow $\Sigma_{\aalpha}$ with a model collection of pairwise disjoint arcs $\{a_{1}, \ldots, a_{n}\}$ which constitute a basis for $H_{1}(\Sigma_{\aalpha};\partial \Sigma_{\aalpha})$, as in Figure \ref{fig:standard-arc-basis} below. We call such a collection the \emph{standard arc basis}, and note that $n$ can be computed as $n = 2p + b - 1$. 

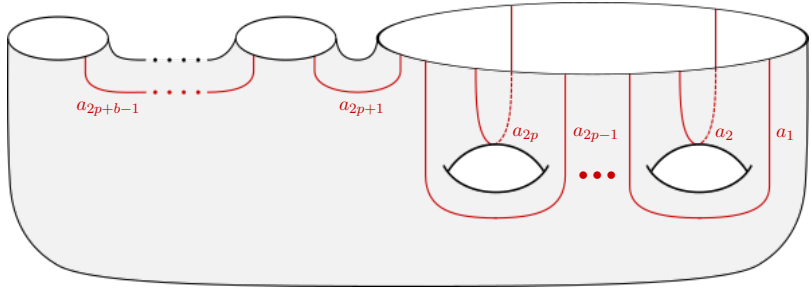
\begin{figure}[ht]
    \centering
    \resizebox{0.65\textwidth}{!}{%
    \begin{tikzpicture}

\node[anchor=south west,inner sep=0] at (0,0) {\includegraphics{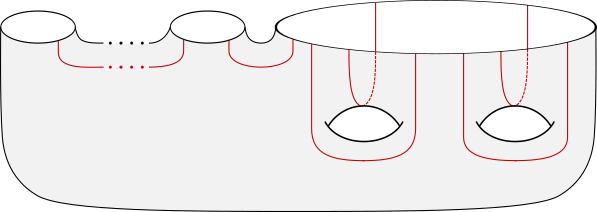}};


\node[label=right:{\color{RED}$a_{1}$}] at (14.9,3){};
\node[label=right:{\color{RED} $a_{2}$}] at (13.7,3){};
\node[label=right:{\color{RED} $a_{2p-1}$}] at (10.9,3){};
\node[label=right:{\color{RED} $a_{2p}$}] at (9.7,3){};
\node[label=above:{\color{RED} $a_{2p+1}$}] at (7,3.1){};
\node[label=above:{\color{RED} $a_{2p+b-1}$}] at (2,3.1){};
\node[label = right:{\color{RED} $ \bullet \bullet \bullet$}] at (11,2.2){};
\end{tikzpicture}
}
\caption{The standard arc basis of $H_{1}(\Sigma_{\aalpha}, \partial \Sigma_{\aalpha}; \Z)$.}
    \label{fig:standard-arc-basis}
\end{figure}

Next, define $\{b_{1}, \ldots, b_{n}\} \subset \Sigma_{\aalpha}$ and $\{c_{1}, \ldots, c_{n}\} \subset \Sigma_{\aalpha}$ to be two additional arc bases which satisfy the following criteria (see Figure \ref{fig:local-arcs}):
\begin{enumerate}
    \item The arc bases $\{a_{1}, \ldots, a_{n}\}$, $\{b_{1}, \ldots, b_{n}\}$ and $\{c_{1}, \ldots, c_{n}\}$ are isotopic (not relative to the endpoints) by a small isotopy;
\item For each $i = 1, \ldots, n$, $a_{i}$ has a single positive transverse intersection with $b_{i}$, where the orientation of $b_{i}$ is inherited from $a_{i}$.
    \item For each $i = 1, \ldots, n$, $b_{i}$ has a single positive transverse intersection with $c_{i}$, where the orientation of the $c_{i}$ is inherited from the $b_{i}$.
    \item For each $i = 1,\ldots, n$, $a_{i}$ has a single positive intersection with $c_{i}$.
\end{enumerate}

\begin{figure}[ht]
    \centering
    \resizebox{0.45\textwidth}{!}{%
     \begin{tikzpicture}

\node[anchor=south west,inner sep=0] at (0,0) {\includegraphics{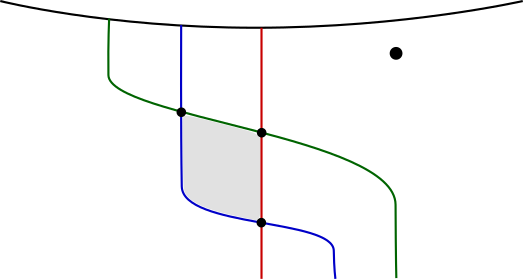}};



\node[label=below:{\Large \color{RED} $a_{i}$}] at (7,0){};
\node[label=below:{\Large \color{BLUE} $b_{i}$}] at (9,0){};
\node[label=below:{\Large \color{GREEN} $c_{i}$}] at (10.5,0){};
\node[label=above:{\Large $w$}] at (10.5, 5){};
\node[label=below:{\Large $\partial \Sigma_{\aalpha}$}] at (1, 7.1){};
\end{tikzpicture}
}
    \caption{A zoomed in picture near the boundary of $\Sigma_{\aalpha}$.}
    \label{fig:local-arcs}
\end{figure}
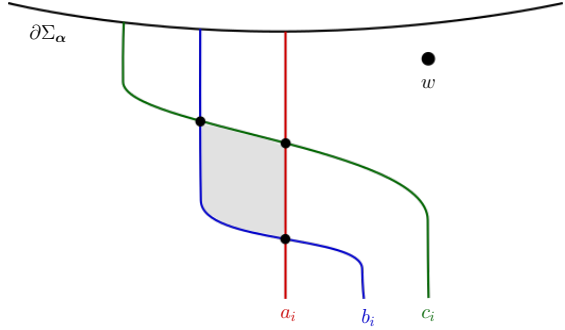
    
Our next step is to use the gradient vector field $\nabla f_{\aalpha}$ of the Morse function $f_{\aalpha}$ to flow the arcs $\{a_{1}, \ldots, a_{n}\} \subset \Sigma_{\aalpha}$ onto $\Sigma$. We'll denote the images of $\{a_{1}, \ldots, a_{n}\}$ under this flow by $\aa = \{\aa_{1}, \ldots, \aa_{n}\} \subset \Sigma$. Note that generic choices ensure that the $\aa_{i}$ are pairwise disjoint form each other and from the original $\aalpha$-curves $\{\aalpha_{1}, \ldots, \aalpha_{g-p}\} \subset \Sigma$. Note, however, that the images $\{\aa_{1}, \ldots, \aa_{n}\}$ are only well-defined up to handle-slides over the original $\aalpha$-curves.

With the data of $(\Sigma, \aalpha, \bbeta, \ggama; \aa)$ in hand (along with the additional data $(g, \langle \cdot, \cdot \rangle, \conn)$ that we started with), we're ready to implement the monodromy algorithm \cite[Theorem 5]{castro2018diagrams} of Gay-Castro-Pinz\'on-Caiced\'o to obtain two new collections of arcs $\bb = \{\bb{1}, \ldots, \bb_{n}\}$ and $\cc = \{\cc_{1}, \ldots, \cc_{n}\}$ which, when taken all together with $\aa$, encode the monodromy diffeomorphism $\mu: \Sigma_{\aalpha} \to \Sigma_{\aalpha}$ of the open book on the boundary $3$-manifold $Y = \partial X$ (see \cite{castro2018diagrams} for more details).

To obtain $\bb$, perform a sequence of handle-slides of $\aa$ arcs over $\aalpha$ curves until $\aa \cap \bbeta = \emptyset$; the resulting collection of arcs is $\bb = \{\bb_{1}, \ldots, \bb_{n}\}$. Next, we obtain $\cc$ by performing another sequence of handle-slides of $\bb$ arcs over $\bbeta$ curves until $\bb \cap \ggama = \emptyset$, and denote the resulting collection of arcs by $\cc = \{\cc_{1}, \ldots, \cc_{n}\}$.  By construction, the data $\D = (\Sigma, \aalpha, \bbeta, \ggama; \aa, \bb, \cc)$ constitute an \emph{arced (relative) trisection diagram} of $X$ \cite[Definition 2.12]{gay2018doubly}.

We now describe how to glue together the above data to construct a Heegaard triple, in the sense of Ozsv\'ath-Szab\'o \cite[Section 8.1]{ozsvath2004holomorphic}, which encodes the cobordism $X: \emptyset \to Y$.  Let $\Sigmabar$ be the surface obtained by gluing the boundaries of $\Sigma$ and $-\Sigma_{\aalpha}$ via an orientation reversing diffeomorphism (see Figure \ref{fig:Four-Ball-Glue} below)
\begin{equation}
    \Sigmabar : = \Sigma \cup_{\partial} -\Sigma_{\aalpha}.
\end{equation}
Note that the genus of $\Sigmabar$ is $g(\Sigmabar) = g + p + b -1$.

\begin{figure}[H]
    \centering
    \includegraphics[scale = 0.17]{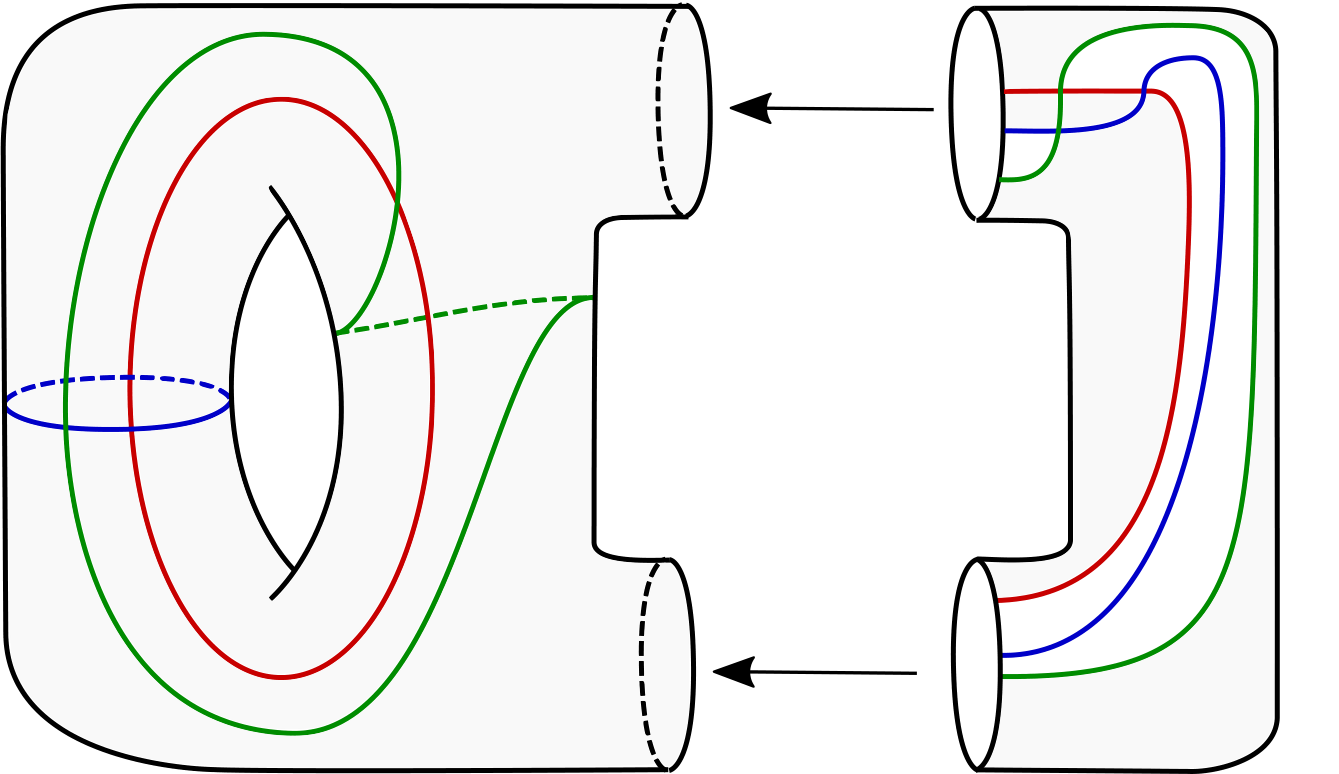}
    \caption{Constructing $\Sigmabar$ by identifying the boundary of the central surface $\Sigma$ (left-hand-side) with that of a page of the open book (right-hand-side) on $Y$ which comes decorated with the parallel arc bases $\{a_{1}\}, \{b_{1}\}$, and $\{c_{1}\}$, drawn in red, blue, and green respectively.}
    \label{fig:Four-Ball-Glue}
\end{figure}

Next, we define three new handlebodies $U_{\aalphabar}, U_{\bbetabar}$, and $U_{\ggamabar}$, each bounded by $\Sigmabar$, by specifying their attaching curves.  The $U_{\aalphabar}$ handlebody is determined by the curves $\{\aalphabar_{1}, \ldots, \aalphabar_{g + p + b -1}\}$ where
\begin{equation}
    \label{eq:alphabar-handlebody}
    \aalphabar_{i} = \begin{cases}
    \alpha_{i} & \quad 1 \leq i \leq |\aalpha| \\
    \aa_{i} \cup_{\partial} \overline{a}_{i} & |\aalpha| + 1 \leq i \leq g(\Sigmabar)
    \end{cases}
\end{equation}

For the $\bbetabar$-handlebody $U_{\bbetabar}$, we define
\begin{equation}
    \label{eq:betabar-handlebody}
    \bbetabar_{i} = \begin{cases}
    \beta_{i} & 1 \leq i \leq |\bbeta| \\
    \bb_{i} \cup_{\partial} \overline{b}_{i} & |\bbeta| + 1 \leq i \leq g(\Sigmabar)
    \end{cases}
\end{equation}

Finally, the $\ggamabar$-handlebody $U_{\ggamabar}$ is determined by
\begin{equation}
    \label{eq:deltabar-handlebody}
    \ggamabar_{i} = \begin{cases}
    \ggama_{i} & 1 \leq i \leq |\ggama| \\
    \cc_{i} \cup_{\partial} \overline{c}_{i} & |\ggama| + 1 \leq i \leq g(\Sigmabar)
    \end{cases}
\end{equation}

\begin{example}
Consider for example the relative trisection diagram for $X = B^{4}$ in the left-hand-side of Figure \ref{fig:Four-Ball-Glue}. After performing the procedure described above, the resulting Heegaard triple looks like Figure \ref{fig:B4-Example} shown below.

\begin{figure}[H]
    \centering
    \includegraphics[scale = 0.17]{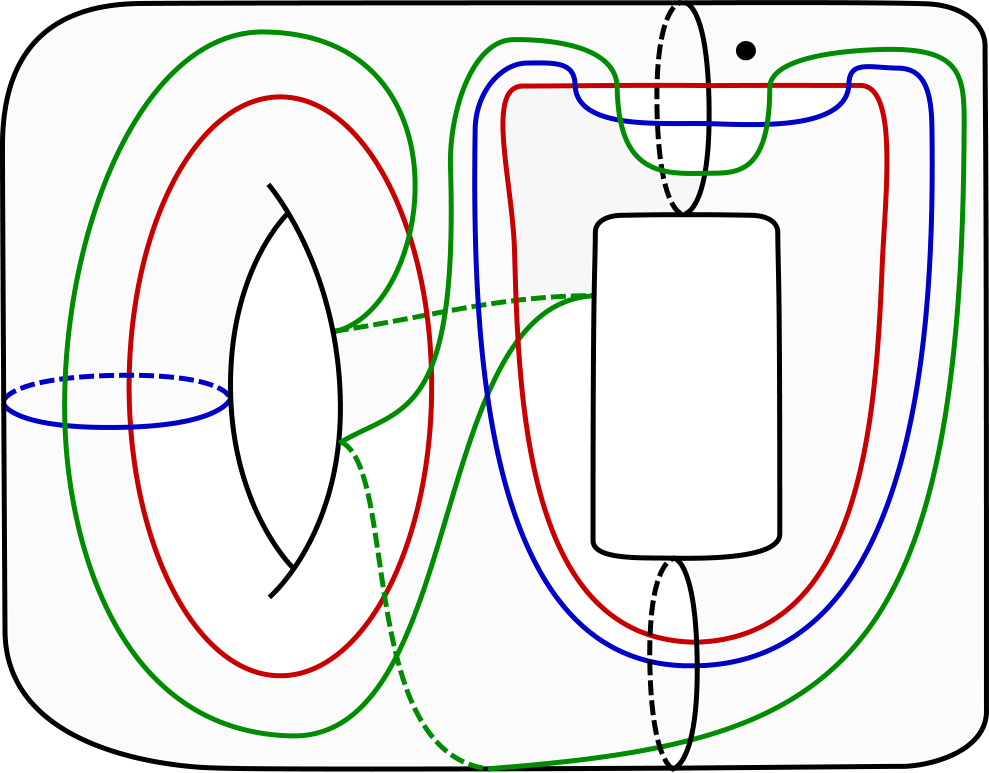}
    \caption{A Heegaard triple produced by applying the procedure described above to the relative trisection diagram for $X = B^{4}$. The closed surface $\Sigmabar$ is obtained by gluing the central surface $\Sigma$ to $-\Sigma_{\aalpha}$ along their boundaries, and the closed curves are obtained by taking a union of the original closed curves from the trisection diagram $(\Sigma, \aalpha, \bbeta, \ggama)$ with those obtained by `doubling' the arc bases using the gradient vector fields of $f_{\aalpha}, f_{\bbeta}$, and $f_{\ggama}$, respectively.}
    \label{fig:B4-Example}
\end{figure}
\end{example}

Thus far, we have described how, given a relative trisection diagram $\D = (\Sigma, \aalpha, \bbeta, \ggama)$ which is compatible with a given $(g,k;p,b)$-trisection map $f:X \to \disk$, to construct a new Heegaard triple $\Dbar = (\Sigmabar, \aalphabar, \bbetabar, \ggamabar)$.  However, it is not at all clear how the original $4$-manifold $X$, as described by the diagram $\D$, and the potentially new $4$-manifold $\Xbar$, as described by the diagram $\Dbar$, are related.  The remainder of this section clarifies this relationship via a technique which we call a \emph{trisector's cut}.

Our strategy for relating $X$ and $\Xbar$ involves a series of intermediate manifolds which we now describe. Starting with $X$, which comes equipped with the decomposition $X = X_{1} \cup X_{2} \cup X_{3}$, consider a collar neighborhood of the boundary of $X_{3}$, denoted $\nu(\partial X_{3})$. 

\begin{figure}[H]
    \centering
    \includegraphics[scale = 0.3]{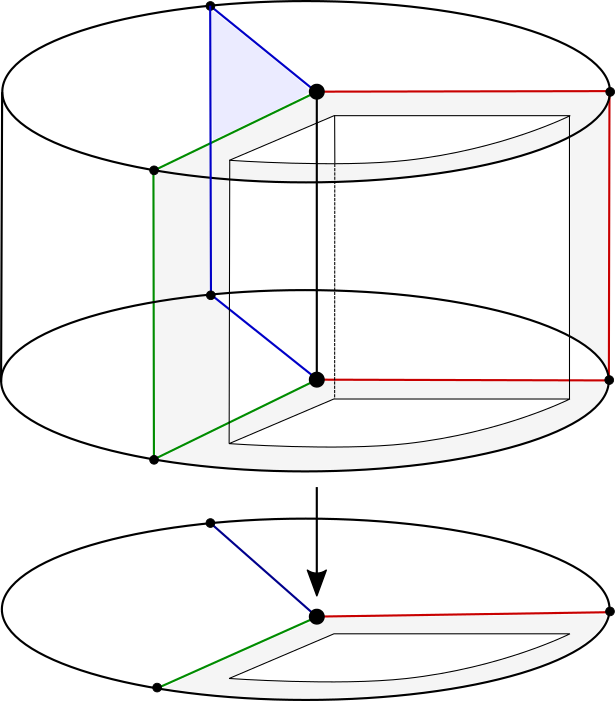}
    \caption{A collar neighborhood of the boundary $X_{3}$.}
    \label{fig:removing-and-drilling}
\end{figure}

After rounding corners we parametrize this collar neighborhood via
\[
\varphi: [0,1] \times \#^{k_{3}}S^{1} \times S^{2} \to \nu(\partial X_{3}),
\]
where $\partial X_{3}$ is embedded in $\nu(\partial X_{3})$ as $\{0\} \times \#^{k_{3}} S^{1} \times S^{2}$. For a chosen basepoint $z \in \pi^{-1}(1) \cong \Sigma_{\aalpha}$, let
\[
\eta: [0,1] \to \nu(\partial X_{3})
\]
be a short arc connecting $z$ to its image in $\{1\} \times \pi^{-1}(1)$. This being done, delete from $X_{3}$ the complement of $\nu(\partial X_{3})$ union a tubular neighborhood of $\eta$. 

\begin{figure}[H]
    \centering
    \includegraphics[scale = 0.3]{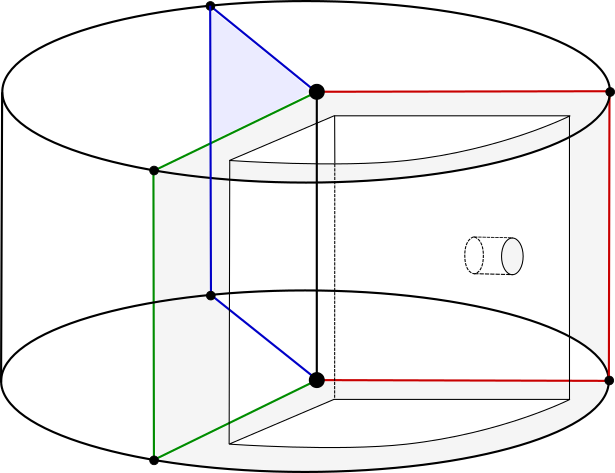}
    \caption{A schematic for deleting the complement of $\nu(\partial X_{3})$ union a tubular neighborhood of $\eta$.}
    \label{fig:Removing-And-Drilling-2}
\end{figure}

In symbols, delete the following subset from $X_{3}$:
\begin{equation}
    \label{eq:delete-from-X3}
    \big(X_{3} \setminus \nu(\partial X_{3})\big) \cup \nu(\eta)
\end{equation}
We give the resulting $4$-manifold a name, $X^{\#}$, and its importance is demonstrated in Proposition \ref{prop:Xsharp} below.

\begin{prop}
\label{prop:Xsharp}
Let $X_{\aalphabar, \bbetabar, \ggamabar}$ be the four-manifold constructed as in equation \eqref{eq:Heegaard-Triple-4-manifold} from the Heegaard triple $(\Sigmabar, \aalphabar, \bbetabar, \ggamabar)$, and define $\Xbar$ to be the smooth four-manifold obtained from $X_{\aalphabar, \bbetabar, \ggamabar}$ after filling in the boundary components $-Y_{\aalphabar, \bbetabar}$ and $-Y_{\bbetabar, \ggamabar}$ with $\natural^{k_{i} + 2p + b - 1} S^{1} \times B^{3}$, $i = 1,2$, respectively. Then the four-manifolds $X^{\#}$ and $\Xbar$ are diffeomorphic.
\end{prop}

\begin{remark}
The author would like to warmly thank David Gay and Juanita Pinz\'on-Caicedo for helpful suggestions during the development of this proof. 
\end{remark}

\begin{proof}
The essential point of the argument is showing how to embed the spine $X_{\aalphabar, \bbetabar, \ggamabar}$ of $\Xbar$ into $X^{\#}$. To do so, we need to identify the surface $\Sigmabar$ and the handlebodies it bounds $U_{\aalphabar}, U_{\bbetabar}$, and $U_{\ggamabar}$ as the appropriate submanifolds of $X^{\#}$. The result will then quickly follow from the uniqueness theorem of Laudenbach-Poenaru \cite{laudenbach1972note}.

To begin, notice that after making the modifications to $X_{3} \subset X$ as in Figure \ref{fig:removing-and-drilling}, the base diagram is now reminiscent of the familiar \emph{keyhole} contour which we parametrize as $B = [-\pi/6, \pi/6] \times [0,1]$ where $\theta \in [-\pi/6, \pi/6]$ and $t \in [0,1]$ are coordinates.

Following Behrens \cite[Section 3.2]{behrens2014smooth}, we say that a parametrization $\kappa: B \to [-\pi/6,\pi/6] \times [0,1]$ is \emph{compatible with $\pi$} if the critical image $C_{\kappa}: = \kappa \circ \pi(\crit(\pi))$ is in the following standard position:
\begin{itemize}
    \item All cusps point to the right (i.e. in the positive $t$-direction).
    \item Each $R_{\theta}:= \{\theta \} \times [0,1]$ meets $C_{\kappa}$ in exactly $g - p$ points, and each intersection is either at a cusp or meets transversely in a fold point.
    \item For a fixed small $\varepsilon > 0$, there exists a $2\varepsilon$-neighborhood $N_{2\varepsilon}$ of $\partial^{\theta} B := [-\pi/6, \pi/6] \times \{0,1\}$ such that $\kappa \circ \pi(\crit(\pi)) \cap N_{2\varepsilon} = \emptyset$.
\end{itemize}

\begin{figure}[ht]
    \centering
    \includegraphics[scale = 0.5]{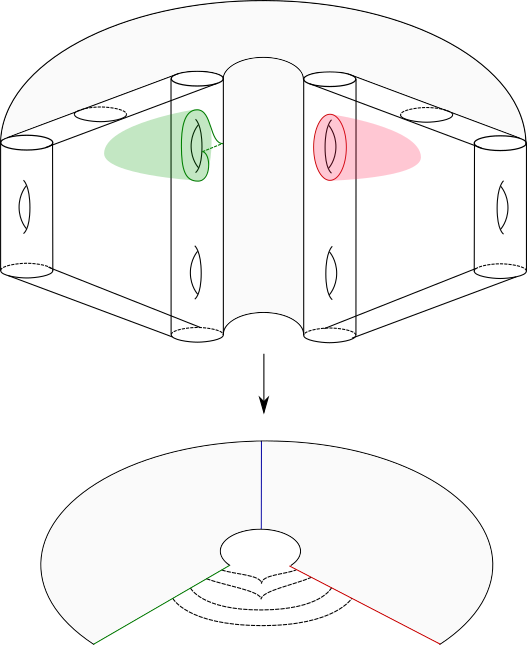}
    \caption{An impressionistic picture of the trisector's cut. The green arc in the base represents $\eta_{\ggama}$, and above it lies the relative compression body $U_{\ggama}$ viewed as a relative cobordism from the central surface $\Sigma$ to $\Sigma_{\ggama}$. Analogous statements can be made for the red arc in the base which represents $\eta_{\aalpha}$.}
    \label{fig:Trisectors-Cut}
\end{figure}

Fix a $\pi$-compatible parametrization $\kappa: B \to [-\pi/6, \pi/6] \times [0,1]$ of the base, and consider the reference arcs $\eta_{\aalpha}:= \{-\pi/3\} \times [0,1]$, $\eta_{\bbeta} := \{\pi\} \times [0,1]$, and $\eta_{\ggama} : = \{\pi/3\} \times [0,1]$. Observe that 
\begin{align*}
    & U_{\aalphabar} : = \pi^{-1}(\eta_{\aalpha}) \\ 
    & U_{\bbetabar} : = \pi^{-1}(\eta_{\bbeta}) \\ 
    & U_{\ggamabar} : = \pi^{-1}(\eta_{\ggama})
\end{align*}
are each relative compression bodies. It's well-known that one can round the corners of these compression bodies and obtain honest $3$-dimensional handlebodies. To be explicit, we refer to Lemma 8.4 of \cite{juhasz2018contact} where the reader can also find a proof.

\begin{lemma}[Lemma 8.4 of \cite{juhasz2018contact}]
\label{lem:rounding-corners}
Let $U_{\aalpha}$ be the relative compression body formed by attaching $3$-dimensional $2$-handles to $I \times \Sigma$ along the curves $\{0\} \times \aalpha$.  After rounding corners, we can view $U_{\aalpha}$ as a handlebody (in the usual sense) of genus $|\aalpha| - \chi(\Sigma_{\aalpha}) + 1$ and boundary
\[
\big(\{1\} \times \Sigma\big) \cup_{\partial} \overline{\Sigma}_{\aalpha}.
\]
Furthermore, a set of compressing disks for $U_{\aalpha}$ can be obtained by taking $|\aalpha|$ compressing disks $D_{\aalpha}$ with boundary $\{1\} \times \aalpha$ for $\alpha \in \aalpha$, as well as disks of the form $D_{c^{\ast}_{i}}: = I \times c_{i}^{\ast}$ for pairwise disjoint, embedded arcs $c^{\ast}_{1}, \ldots, c^{\ast}_{b_{1}\Sigma_{\aalpha}}$ in $\Sigma$ that avoid the $\aalpha$ curves, and form a basis of $H_{1}(\Sigma_{\aalpha}, \partial \Sigma_{\aalpha})$.  These cut $U_{\aalpha}$ into a single $3$-ball. 
\end{lemma}

Applying Lemma \ref{lem:rounding-corners} to the three sutured compression bodies $U_{\aalphabar}, U_{\bbetabar}$, and $U_{\ggamabar}$ above, we obtain three $3$-dimensional handlebodies with $\partial U_{\ttau} = \Sigmabar_{\ttau}$ for each $\ttau \in \{\aalphabar, \bbetabar, \ggamabar\}$. We take as the central surface in our spine-decomposition of $X^{\#}$ to be $\Sigmabar:= \partial U_{\aalphabar}$. Clearly, $\Sigmabar$ bounds the $U_{\aalphabar}$ handle-body described in equation \eqref{eq:alphabar-handlebody}. Notice, however, that the $U_{\bbetabar}$ and $U_{\ggamabar}$ handlebodies are completely disjoint from $\Sigmabar$. To remedy this, we isotope the attaching circles for the $\bbetabar$- and $\ggamabar$-handlebodies onto $\Sigmabar$, and it is via this isotopy that we see how the monodromy of the open book decomposition of $Y$ naturally arises. After isotoping the attaching curves onto the same central surface $\Sigmabar$, we will have completed the proof that the spine $X_{\aalphabar, \bbetabar, \ggamabar}$ of $\Xbar$ embeds into $X^{\#}$.

Now, we'll construct an isotopy for the attaching circles for the handlebodies $U_{\bbetabar}$ and $U_{\ggamabar}$. To do so, recall from the beginning of this section that we have a chosen $\pi$-compatible connection $\conn$. The first step is to thicken the surface $\Sigmabar_{\bbeta} : = \partial U_{\bbetabar}$ to $\Sigmabar_{\bbeta} \times [0, 2\varepsilon]$ using the inward pointing normal direction coming from the boundary.  Since we have an $\pi$-compatible parametrization of the base, the attaching circles on $\Sigmabar_{\bbeta} \times \{2\varepsilon\}$ are isotopic to those of $\Sigmabar_{\bbetabar} = \Sigmabar_{\bbetabar} \times \{0\}$. Next, we use the $\pi$-compatible connection $\conn$ to transport the attaching circles on $\Sigmabar_{\bbetabar} \times \{2\varepsilon\}$ onto to $\Sigmabar_{\aalpha} \times \{2\varepsilon\}$.

\begin{figure}[ht]
    \centering
    \includegraphics[scale = 0.4]{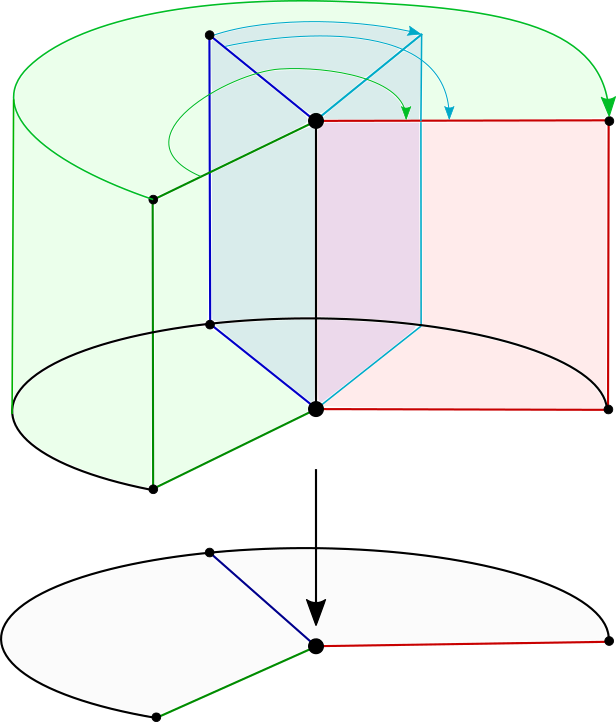}
    \caption{A schematic for visualizing how the attaching curves for $U_{\bbetabar}$ and $U_{\ggamabar}$ are isotoped onto $\Sigmabar$ using parallel transport. On the right, there is a filled-in red square which is meant to represent the relative compression body $U_{\aalphabar}$ and its boundary is $\Sigmabar$--similarly for the blue and green squares. Once the blue square is `pushed in', one can flow the attaching curves clockwise onto a copy of $\Sigmabar$ which is also slightly pushed in. Finally, one employs the same strategy to the $\ggamabar$ curves. The general impression should be that of a nautilus shell.}
    \label{fig:full-flow}
\end{figure}

Since the $\bbetabar$ attaching circles have been isotoped onto $\Sigmabar$ using a $\pi$-compatible connection, it follows by the monodromy algorithm of \cite[Theorem 5]{castro2018diagrams} that the resulting curves are precisely those for $U_{\bbetabar}$--see Figure \ref{fig:full-flow} below. Next, we repeat the above process using the $\ggamabar$ attaching circles and arrive at the same conclusion for $U_{\ggamabar}$. Thus, we've shown that the spine $X_{\aalphabar, \bbetabar, \ggamabar}$ of $\Xbar$ embeds into $X^{\#}$, and the proposition follows after applying the uniqueness theorem of \cite{laudenbach1972note} to the remaining boundary components.
\end{proof}

\begin{remark}
\label{cor:Xsharp-X}
The boundary of $X^{\#}$ is $Y \# (\#^{k_{3}} S^{1} \times S^{2})$, and after filling in the $\#^{k_{3}}S^{1} \times S^{2}$, we recover the original $4$-manifold $X$. 
\end{remark}

\begin{cor}
\label{cor:alpha-beta-gamma-curves}
In the Heegaard triple $(\Sigmabar, \aalphabar, \bbetabar, \ggamabar, w)$ constructed above, we have that $(\Sigmabar, \aalphabar, \bbetabar)$, $(\Sigmabar, \bbetabar, \ggamabar)$ and $(\Sigmabar, \aalphabar, \ggamabar)$ are Heegaard diagrams for the three-manifolds $\#^{\ell_{1}} S^{1} \times S^{2}$, $\#^{\ell_{2}}S^{1} \times S^{2}$, and $Y \# (\#^{k_{3}}S^{1} \times S^{2})$ where $\ell_{i} = k_{i} + 2p + b - 1$.
\end{cor}

\begin{proof}
The statements for $(\Sigmabar, \aalphabar, \bbetabar)$ and $(\Sigmabar, \bbetabar, \ggamabar)$ follow from a combination of two facts; the first being that $(\Sigma, \aalpha, \bbeta, \ggama)$ is a relative trisection, so that to begin with the pairwise tuples yield connect sums of $S^{1} \times S^{2}$; and the second being that the monodromy of the open book can be trivialized over one sector at a time. 
\end{proof}

\subsection{Holomorphic triangles and cobordism maps}
\label{subsec:hol-triangle-cob-map}

Fix $X$ to be a smooth, oriented, compact four-manifold with connected boundary, and equip $X$ with a $(g,k;p.b)$-trisection map $\pi: X \to \disk$, a metric $\langle \cdot, \cdot \rangle$, and a $\pi$-compatible connection $\conn$ as in Section \ref{subsec:Heeg-Triples} above. If $(\Sigma, \aalpha, \bbeta, \ggama)$ is a (relative) trisection diagram associated to these data, we show how the holomorphic triangle map \eqref{eq:hol-triangles} applied to the pointed Heegaard triple $\underline{\H} = (\Sigmabar, \aalphabar,\bbetabar, \ggamabar, w)$ computes the induced cobordism map of Ozsv\'ath and Szab\'o.  

\begin{prop}
\label{lem:strongly-equivalent}
Let $\underline{\H} = (\Sigmabar, \aalphabar, \bbetabar, \ggamabar, w)$ be a pointed Heegaard triple constructed using the prescription described in subsection \ref{subsec:Heeg-Triples} above, and let $X_{\aalphabar, \bbetabar, \ggamabar}$ be its associated four-manifold spine which we view as a cobordism from $\#^{\ell_{1}}S^{1} \times S^{2}$ to $Y^{\#}$ after filling in $-Y_{\bbetabar, \ggamabar}$ with $\natural^{\ell_{2}}S^{1} \times B^{3}$ and a $4$-handle. Then $\underline{\H}$ is slide-equivalent to another Heegaard triple $\H'$ which is subordinate to a bouquet for a framed link $\L \subset \#^{\ell_{1}}S^{1} \times S^{2}$ for which the $2$-handle cobordism $W(\#^{\ell_{1}}S^{1} \times S^{2}, \L)$ is diffeomorphic to $X_{\aalphabar, \bbetabar, \ggamabar}$ as cobordisms from $\#^{\ell_{1}}S^{1} \times S^{2}$ to $Y^{\#}$.
\end{prop}

\begin{proof}
Recall from \cite[Definition 4.5]{meier2016classification} that a disk $D_{\gamma}$ properly embedded in $U_{\ggamabar}$ is \emph{primitive in $U_{\ggamabar}$ with respect to $U_{\bbetabar'}$} if there exists a compression disk $D_{\beta_{i}'}$ satisfying the condition $|D_{\gamma} \cap D_{\beta_{i}'}| = 1$. Since $(\Sigma, \bbetabar', \ggamabar)$ is a genus $\underline{g} = g + p + b -1$ Heegaard diagram for $\#^{\ell_{2}}S^{1} \times S^{2}$, it follows from \cite[Theorem 2.7]{meier2016classification} that $U_{\ggamabar}$ admits an ordered collection of compression disks $\{D_{\ggamabar_{i}'}\}$ where the corresponding attaching circles $\ggamabar_{i}' = \partial D_{\ggamabar_{i}'}$ satisfy
\begin{enumerate}
    \item For $i = 1, \ldots, g - k -p$, $\ggamabar_{i}'$ satisfies $|\ggamabar_{i}' \cap \bbetabar_{i}'| = 1$ and $|\ggamabar_{i}' \cap \bbetabar_{j}| = 0$ for $i \neq j$.
    \item For $i = g - k - p + 1, \ldots, g + p + b -1$, $\ggamabar_{i}'$ is parallel to $\bbetabar_{i}'$. 
\end{enumerate}
We remark that since $\ggamabar'$ and $\ggamabar$ are cut systems for the same handlebody $U_{\ggamabar}$, it follows from \cite{johannson2006topology} that $\ggamabar \sim \ggamabar'$.

This being done, it follows from \cite[p.5]{kim2020trisections} (see, in particular \cite[Figure 2]{kim2020trisections}) that for $i = 1, \ldots, g - k - p$, $\ggamabar_{i}'$ can be interpreted as the framed attaching sphere for a $2$-handle cobordism, where each $\ggamabar_{i}'$ is given the surface framing.

Finally, we exhibit a bouquet for the framed attaching link $\L = \{\ggamabar_{1}', \ldots, \ggamabar_{g-p-k}'\}$ and check that $(\Sigmabar, \aalphabar', \bbetabar', \ggamabar')$ is subordinate to it. For each $\ggamabar_{i}' \in \{\ggamabar_{1}', \ldots, \ggamabar_{g-k-p}'\}$, choose a properly embedded arc $\eta_{i} \subset U_{\bbetabar}$ which has one endpoint on $\ggamabar_{i}'$ and the other on $w$, the fixed basepoint. Then the union of $\eta_{i}$ comprise a bouquet for the link $\L$. Furthermore, $(\Sigmabar, \{\aalphabar_{1}, \ldots, \aalphabar_{g + p + b -1}'\}, \{\bbetabar_{g-p-k+1}', \ldots, \bbetabar_{g + p + b -1}'\}$ is a Heegaard diagram for the complement of $\L$ in $\#^{\ell_{1}}S^{1} \times S^{2}$. Next, taking a thin tubular neighborhood of $\bbetabar_{i}' \cup \ggamabar_{i}'$ constitutes a punctured torus for each $i = 1, \ldots, g - k - p$. Last, the conditions that $\bbetabar_{i}'$ constitute a meridian and that $\ggamabar_{i}'$ constitute a longitude are self evident after using the surface framing to push $\ggamabar_{i}'$ into $U_{\bbetabar}$ handlebody. Thus, the conditions \ref{itm:bouq-1} -- \ref{itm:bouq-5} are satisfied.
\end{proof}

The remainder of the proof of Theorem \ref{thm:Main} is an application of various naturality results which are standard in the Heegaard Floer theory. We recapitulate some of the details here, but we claim no originality to them--see, for example, \cite[p.360]{ozsvath2004holomorphic} for what is essentially the same argument and \cite{juhasz2012naturality} for more details concerning naturality issues in Heegaard Floer homology. Following \cite{juhasz2012naturality}, we make a notational definition. 

\begin{definition}
Let $(\Sigma, \aalpha, \bbeta, \bbeta')$ be an admissable triple diagram. If $\bbeta \sim \bbeta'$ are handle-slide equivalent, then we'll write $\Psi^{\aalpha}_{\bbeta \to \bbeta'}$ for the map
\begin{equation}
    F^{\circ}_{\aalpha, \bbeta, \bbeta'}(- \otimes \Theta_{\bbeta, \bbeta'}): HF^{\circ}(\Sigma, \aalpha, \bbeta) \to HF^{\circ}(\Sigma, \aalpha, \bbeta')
\end{equation}
Similarly, if $\aalpha' \sim \aalpha$, then let $\Psi^{\aalpha' \to \aalpha}_{\ggama}$ denote the map
\begin{equation}
    F^{\circ}_{\aalpha', \aalpha, \ggama}(\Theta_{\aalpha', \aalpha} \otimes -): HF^{\circ}(\Sigma, \aalpha', \ggama) \to HF^{\circ}(\Sigma, \aalpha, \ggama)
\end{equation}
\end{definition}

We take a moment to compare $\spinc$-structures on $X$ to those on $X^{\#}$. Observe that there is a natural restriction map
\begin{equation}
    \label{eq:restrict-spinc}
    r: \spinc(X) \to \spinc(X^{\#})
\end{equation}
The restriction map $r$ is surjective, and conversely, a $\spinc$-structure $\ss^{\#}$ on $X^{\#}$ admits a unique extension to $X$ if it is isomorphic to the unique torsion $\spinc$-structure $\ss_{0}$ in a neighborhood of $\#^{k_{3}}S^{1} \times S^{2}$. 

\begin{prop}
\label{prop:triangle-naturality}
Fix a $\spinc$-structure $\ss \in \spinc(X^{\#})$. Let $\H = (\Sigmabar, \aalphabar, \bbetabar, \ggamabar, w)$ be the pointed $\ss$-admissable Heegaard triple constructed as above, and let $\H' = (\Sigmabar, \aalphabar', \bbetabar', \ggamabar', w)$ be a Heegaard triple which is strongly equivalent to $\H$ and which is subordinate to a bouquet for a framed link $\L$ as in Proposition \ref{lem:strongly-equivalent} above. Then in the diagram below 
\begin{figure}[H]
    \centering
    \begin{tikzcd}
    HF^{\circ}(\Sigmabar, \aalphabar, \bbetabar, \ss_{0}) \ar{d}{\Psi_{\bbetabar \to \bbetabar'}^{\aalphabar \to \aalphabar'}} \ar{r}{F^{\circ}_{\aalphabar, \bbetabar, \ggamabar, \ss}} & HF^{\circ}(\Sigmabar, \aalphabar, \ggamabar, \ss_{\aalphabar, \ggamabar}) \ar{d}{\Psi_{\ggamabar \to \ggamabar'}^{\aalphabar \to \aalphabar'}} \\
    HF^{\circ}(\Sigmabar, \aalphabar', \bbetabar', \ss_{0}) \ar{r}{F^{\circ}_{\L, \ss}} & HF^{\circ}(\Sigmabar, \aalphabar', \ggamabar', \ss_{\aalphabar', \ggamabar'}) 
    \end{tikzcd}
    \label{fig:naturality-diagram}
\end{figure}
\noindent
we have the following equality
\begin{equation}
\label{eq:naturality}
    F^{\circ}_{\L, \ss} \circ \Psi_{\bbetabar \to \bbetabar'}^{\aalphabar \to \aalphabar'} (\Theta_{\aalphabar, \bbetabar}) = \Psi_{\ggamabar \to \ggamabar'}^{\aalphabar \to \aalphabar'} \circ F_{\aalphabar, \bbetabar, \ggamabar, \ss} (\Theta_{\aalphabar, \bbetabar})
\end{equation}
\end{prop}

\begin{proof}
Similar results are common in the literature, so we'll be brief (cf. \cite[p.360]{ozsvath2006holomorphic}). By assumption, the cut systems $\aalphabar \sim \aalphabar'$, $\bbetabar \sim \bbetabar'$, and $\ggamabar \sim \ggamabar'$ are related by sequences of isotopies and handleslides. Start by considering the sequence $\aalphabar \sim \aalphabar'$, which yields the following diagram:
\begin{figure}[H]
    \centering
    \begin{tikzcd}[column sep = 1in]
    HF^{\circ}(\Sigmabar, \aalphabar, \bbetabar, \ss_{0}) \ar{r}{F^{\circ}_{\aalphabar, \bbetabar, \ggamabar, \ss}} \ar{d}{\Psi^{\aalphabar \to \aalphabar'}_{\bbetabar}} & HF^{\circ}(\Sigmabar, \aalphabar, \ggamabar, \ss_{\aalphabar, \ggamabar}) \ar{d}{\Psi_{\ggamabar}^{\aalphabar \to \aalphabar'}} \\ 
        HF^{\circ}(\Sigmabar, \aalphabar', \bbetabar, \ss_{0}) \ar{r}[swap]{F^{\circ}_{\aalphabar', \bbetabar, \ggamabar, \ss}} & HF^{\circ}(\Sigmabar, \aalphabar', \ggamabar, \ss_{\aalpha', \ggama}) 
    \end{tikzcd}
    \caption{The commutative square associated to the sequence of isotopies and handle slides connecting $\aalphabar$ to $\aalphabar'$.}
    \label{fig:aalpha-chain-homotopy}
\end{figure}

\noindent
By \cite[Proposition 9.10]{juhasz2012naturality} we have that both $\Psi_{\bbetabar}^{\aalphabar \to \aalphabar'}$ and $\Psi_{\ggamabar}^{\aalphabar \to \aalphabar'}$ are isomorphisms, and by \cite[Lemma 9.4]{juhasz2012naturality} we have that $HF^{\circ}_{\text{top}}(\Sigmabar, \aalphabar, \bbetabar, \ss_{0}) \cong \F\langle \Theta_{\aalphabar, \bbetabar}\rangle$ and $HF^{\circ}_{\text{top}}(\Sigmabar, \aalphabar', \bbetabar, \ss_{0}) \cong \F\langle \Theta_{\aalphabar', \bbetabar}\rangle$. It is now immediate that $\Psi_{\bbetabar}^{\aalphabar \to \aalphabar'}(\Theta_{\aalphabar, \bbetabar}) = \Theta_{\aalphabar', \bbetabar}$.

Using \cite[Lemma 9.5]{juhasz2012naturality}, we may assume that $(\Sigma, \aalphabar', \aalphabar, \bbetabar, \ggamabar, w)$ has also been made admissable, so we can apply the associativity theorem for holomorphic triangles \cite[Theorem 8.16]{ozsvath2004holomorphic} and conclude that
\begin{equation}
\label{eq:aalpha-commutes}
    F^{\circ}_{\aalphabar', \aalphabar, \ggamabar}\big(\Theta_{\aalphabar', \aalphabar} \otimes F^{\circ}_{\aalphabar, \bbetabar, \ggamabar, \ss}(\Theta_{\aalphabar, \bbetabar} \otimes \Theta_{\bbetabar, \ggamabar})\big) = F^{\circ}_{\aalphabar', \bbetabar, \ggamabar, \ss}\big(F^{\circ}_{\aalphabar', \aalphabar, \bbetabar}(\Theta_{\aalphabar', \aalphabar} \otimes \Theta_{\aalphabar, \bbetabar}) \otimes \Theta_{\bbetabar, \ggamabar} \big)
\end{equation}
Clearly, equation \eqref{eq:aalpha-commutes} shows that the diagram in Figure \ref{fig:aalpha-chain-homotopy} commutes for the generator $\Theta_{\aalphabar, \bbetabar}$.

Having handled the sequence $\aalphabar \sim \aalphabar'$, we consider next the sequence of isotopies and handleslides amongst the $\bbetabar$-curves. In a similar fashion, we consider the following diagram
\begin{figure}[H]
    \centering
    \begin{tikzcd}[column sep = 1in]
    HF^{\circ}(\Sigmabar, \aalphabar', \bbetabar, \ss_{0}) \ar{r}{F^{\circ}_{\aalphabar', \bbetabar, \ggamabar, \ss}} \ar{d}{\Psi^{\aalphabar'}_{\bbetabar \to \bbetabar'}} & HF^{\circ}(\Sigmabar, \aalphabar', \ggamabar, \ss_{\aalphabar', \ggamabar}) \ar[equal]{d} \\ 
        HF^{\circ}(\Sigmabar, \aalphabar', \bbetabar', \ss_{0}) \ar{r}[swap]{F^{\circ}_{\aalphabar', \bbetabar', \ggamabar, \ss}} & HF^{\circ}(\Sigmabar, \aalphabar', \ggamabar, \ss_{\aalpha', \ggama}) 
    \end{tikzcd}
    \caption{The commutative square associated to the sequence of isotopies and handle slides connecting $\bbetabar$ to $\bbetabar'$.}
    \label{fig:bbeta-chain-homotopy}
\end{figure}
The proof that Figure \ref{fig:bbeta-chain-homotopy} is commutative, however, is slightly different than that for Figure \ref{fig:aalpha-chain-homotopy}, so we include the proof here. As before, we apply \cite[Lemma 9.5]{juhasz2012naturality} to justify that $(\Sigma, \aalphabar', \bbetabar, \bbetabar', \ggamabar, w)$ is admissable. Applying the associativity theorem for holomorphic triangles, we see that
\begin{equation}
    \label{eq:bbeta-commutes}
    F^{\circ}_{\aalphabar', \bbetabar', \ggamabar}\big( F^{\circ}_{\aalphabar', \bbetabar, \bbetabar'}(\Theta_{\aalphabar', \bbetabar} \otimes \Theta_{\bbetabar, \bbetabar'}) \otimes \Theta_{\bbetabar', \ggamabar} \big) = F^{\circ}_{\aalphabar', \bbetabar, \ggamabar}\big(\Theta_{\aalphabar', \bbetabar} \otimes F^{\circ}_{\bbetabar, \bbetabar', \ggamabar}(\Theta_{\bbetabar', \bbetabar} \otimes \Theta_{\bbetabar',\ggamabar}) \big)
\end{equation}
By again applying \cite[Proposition 9.10]{juhasz2012naturality} and \cite[Lemma 9.4]{juhasz2012naturality}, we observe that
\begin{equation}
    \label{eq:bbeta-square}
F^{\circ}_{\bbetabar, \bbetabar', \ggamabar}(\Theta_{\bbetabar, \bbetabar'} \otimes \Theta_{\bbetabar', \ggamabar}) = \Theta_{\bbetabar, \ggamabar}
\end{equation}
which turns equation \eqref{eq:bbeta-square} into
\begin{equation}
    \label{eq:final-bbeta}
    F^{\circ}_{\aalphabar', \bbetabar', \ggamabar}\big( F^{\circ}_{\aalphabar', \bbetabar, \bbetabar'}(\Theta_{\aalphabar', \bbetabar} \otimes \Theta_{\bbetabar, \bbetabar'}) \otimes \Theta_{\bbetabar', \ggamabar} \big) = F^{\circ}_{\aalphabar', \bbetabar, \ggamabar}(\Theta_{\aalphabar', \bbetabar} \otimes \Theta_{\bbetabar, \ggamabar})
\end{equation}
It is immediate from equation \eqref{eq:final-bbeta} that Figure \ref{fig:bbeta-chain-homotopy} commutes for the generator $\Theta_{\aalphabar', \bbetabar}$.

Having studied the sequences $\aalphabar \sim \aalphabar'$ and $\bbetabar \sim \bbetabar'$, we leave it to the reader to build an analogous commutative diagram for the sequence $\ggamabar \sim \ggamabar'$ and top generator $\Theta_{\aalphabar', \ggamabar}$. The proof that it is commutative follows as for the sequence $\bbetabar \sim \bbetabar'$.

To demonstrate the assertion made in the proposition, we observe that after stacking Figures \ref{fig:aalpha-chain-homotopy} and \ref{fig:bbeta-chain-homotopy} on top of the appropriate diagram for the $\ggamabar \sim \ggamabar'$ sequence, we arrive at a new commutative diagram which is equivalent to equation \eqref{eq:naturality}. This is so for two reasons: first, by Definition the maps $F^{\circ}_{\aalphabar', \bbetabar', \ggamabar', \ss}$ and $F^{\circ}_{\L, \ss}$ are equivalent, and second, by \cite[Proposition 9.10]{juhasz2012naturality} we have
\[
\Psi^{\aalphabar \to \aalphabar'}_{\bbetabar \to \bbetabar'} = \Psi^{\aalphabar'}_{\bbetabar \to \bbetabar'} \circ \Psi^{\aalphabar \to \aalphabar'}_{\bbetabar} \qquad \text{ and } \qquad \Psi^{\aalphabar \to \aalphabar'}_{\ggamabar \to \ggamabar'} = \Psi^{\aalphabar'}_{\ggamabar \to \ggamabar'} \circ \Psi^{\aalphabar \to \aalphabar'}_{\ggamabar}.
\]
\end{proof}

\begin{theorem}
\label{thm:rel-invt}
In the diagram below,
\begin{figure}[H]
    \centering
    \begin{tikzcd}[column sep = 1in]
    HF^{\circ}(S^{3}) \ar{d}{F_{1}} \ar{r}{F^{\circ}_{X, \ss}} & HF^{\circ}(Y, \ss) \\
    HF^{\circ}(\Sigmabar, \aalphabar, \bbetabar, \ss_{0}) \ar{r}[swap]{F^{\circ}_{\aalphabar, \bbetabar, \ggamabar, \ss_{\aalphabar, \bbetabar, \ggamabar}}} & HF^{\circ}(\Sigmabar, \aalphabar, \ggamabar, \ss_{\aalphabar, \ggamabar}) \ar{u}{F_{3}} 
    \end{tikzcd}
    \label{fig:main-relative-invariant-diagram}
\end{figure}
\noindent
the following equality holds
\begin{equation}
    \label{eq:main-relative-invariant-equality}
    F_{3} \circ F^{\circ}_{\aalphabar, \bbetabar, \ggamabar, \ss_{\aalphabar, \bbetabar, \ggamabar}}\circ F_{1} (\Theta) = F^{\circ}_{X, \ss}(\Theta)
\end{equation}
\end{theorem}

\begin{proof}
This follows immediately after combining the construction of $X^{\#}$ with Proposition \ref{prop:triangle-naturality}, the definitions of the $1$- and $3$-handle cobordism maps, and the classic results of \cite{ozsvath2006holomorphic} which show that $F_{X, \ss}$ is independent of the handle decomposition of $X$.
\end{proof}

\subsection{Remarks on the contact class}
\label{subsec:contact-class-future-work}

Fix $X$ to be a smooth, oriented, compact four-manifold with connected boundary $\partial X =Y$. As discussed in Section \ref{subsec:relative-trisection-diagrams}, a $(g,k;p,b)$-trisection map $\pi: X \to \disk$ induces an open book decomposition on its boundary $3$-manifold $Y$. Given the data of such an open book, Honda-Kazez-Matic \cite{honda2009contact} define a class $c(\xi) \in HF^{+}(-Y, \ss_{\xi})$ and show that $c^{+}(\xi)$ agrees with the Ozsv\'ath-Szab\'o contact invariant \cite{ozsvath2005heegaard} associated to $(Y, \xi)$, where $\xi$ is a contact structure supported by the given open book. In this section, we initiate a study of the relationship between $c^{+}(\xi)$ and relative trisection maps $\pi$ inducing an open book which supports $\xi$.

To begin, fix a $(g,k;p.b)$-trisection map $\pi: X \to \disk$, and let $(\Sigma, \aalpha, \bbeta, \ggama)$ be its associated diagram. Next, construct the pointed Heegaard triple\footnote{We have intentionally flipped the roles of $\aalphabar$, $\bbetabar$, and $\ggamabar$ in this construction, as will be apparent momentarily.} $(\Sigmabar, \bbetabar, \ggamabar, \aalphabar, w)$ as in Subsection \ref{subsec:Heeg-Triples} above. Following \cite[Section 2.2]{baldwin2013capping}, define for each $i = g-p+1, \ldots, g(\Sigmabar) = g + p + b - 1$ the intersection points $\theta_{i}$, $x_{i}$, and $y_{i}$, as shown in Figure \ref{figure:local-contact-class} below,
\begin{align}
    & \theta_{i} = \bbeta_{i} \cap \ggama_{i} \cap \overline{\Sigma}_{\aalpha} \nonumber \\
    & x_{i} = \ggama_{i} \cap \aalpha_{i} \cap \overline{\Sigma}_{\aalpha} \label{eq:corner-points} \\
    & y_{i} = \bbeta_{i} \cap \aalpha_{i} \cap \overline{\Sigma}_{\aalpha} \nonumber
\end{align}
and let $\Theta$, $\xx$, and $\yy$ be the corresponding intersection points
\begin{align}
    & \Ttheta = \{\Theta_{\bbeta, \ggama}^{(1)}, \ldots, \Theta_{\bbeta, \ggama}^{(g-p)}, \theta_{g-p+1}, \ldots, \theta_{g+p + b - 1}\} \nonumber \in \tbbar \cap \tgbar \\
    & \xx = \{\Theta_{\ggama, \aalpha}^{(1)}, \ldots, \Theta_{\ggama, \aalpha}^{(g-p)}, x_{g-p+1}, \ldots, x_{g + p + b - 1}\} \label{eq:intersection-points} \in \tgbar \cap \tabar \\
    & \yy = \{\Theta_{\bbeta, \aalpha}^{(1)}, \ldots, \Theta_{\bbeta, \aalpha}^{(g-p)}, y_{g-p+1}, \ldots, y_{g + p + b - 1}\} \in \tbbar \cap \tabar \nonumber
\end{align}
To describe the symbols $\Theta_{\xxi, \zzeta}^{(i)}$, for $i = 1, \ldots, g-p$ and $\xxi, \zzeta \in \{\aalpha, \bbeta, \ggama\}$, recall that by the connect sum formula \cite{ozsvathPropsApps} and Corollary \ref{cor:alpha-beta-gamma-curves}, it follows that
\begin{equation}
    \label{eq:connect-sum-ba-gb}
HF^{+}(\Sigmabar, \bbetabar, \aalphabar, \ss_{0}) \cong HF^{+}(\Sigmabar, \ggamabar, \bbetabar, \ss_{0}) \cong \Lambda^{\ast}(H_{1}(\#^{k + 2p + b -1}S^{1} \times S^{2})) \otimes \F[U, U^{-1}]/U\cdot \F[U]
\end{equation}
and
\begin{equation}
    \label{eq:connect-sum-Y-ag}
    HF^{+}(\Sigmabar, \ggamabar, \aalphabar, \ss \# \ss_{0}) \cong HF^{+}(Y, \ss) \otimes HF^{+}(\#^{k}S^{1} \times S^{2};\ss_{0})
\end{equation}
With these observations in mind, we choose the $\Theta_{\xxi, \zzeta}^{(i)}$ so that they represent the top-degree homology class in these decompositions.

\begin{figure}[H]

    \centering
    \resizebox{0.5\textwidth}{!}{%
  \begin{tikzpicture}

\node[anchor=south west,inner sep=0] at (0,0) {\includegraphics{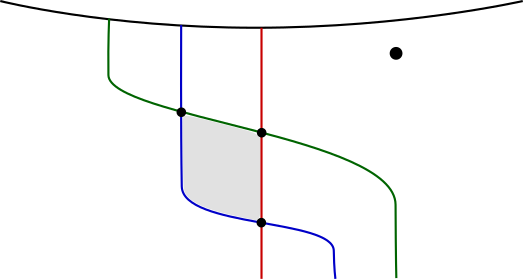}};



\node[label=below:{$w$}] at (10,6){};
\node[label=right:{$x_{i}$}] at (7,4){};
\node[label=left:{$y_{i}$}] at (7,1){};
\node[label = left:{$\theta_{i}$}] at (4.9,4){};


\node[label = below:{$\overline{\Sigma}_{\bm{\alpha}}$}] at (8, 6.5){};
\node[label = above:{$\Sigma$}] at (8,6.6){};
\end{tikzpicture}
}
\caption{A local picture of the intersection points $\theta_{i}$, $x_{i}$, and $y_{i}$.}
\label{figure:local-contact-class}
\end{figure}

Given the above familiar setting, we'd like to make a few remarks:
\begin{itemize}
    \item As in \cite{honda2009contact}, the generator $[\xx, 0]$ is a cycle in $CF^{+}(\Sigmabar, \ggamabar, \aalphabar, w)$, and its image in homology is mapped to $c^{+}(Y, \xi) \in HF^{+}(-Y, \ss_{\xi})$ under the $3$-handle cobordism map. That is,
    \begin{figure}[H]
    \centering
\begin{tikzcd}[row sep = 0.05in]
    HF^{+}(-Y^{\#}, \ss_{\xi} \# \ss_{0}) \ar{r}{F_{3}} & HF^{+}(-Y, \ss_{\xi}) \\
    \left[\xx,0\right] \ar{r} & c^{+}(\xi)
\end{tikzcd}
\end{figure}
\noindent
In particular, the image of $c^{+}(Y, \xi)$ under $F^{+}_{\overline{X}, \ss}$ coincides with the image of $[\xx,0]$ in homology under the map $F^{+}_{\bbetabar, \ggamabar, \aalphabar}(\Theta_{\bbetabar, \ggamabar} \otimes -)$.

\item As in \cite[Proposition 2.3]{baldwin2013capping}, there may be some usefulness in the way $(\Sigmabar, \bbetabar, \ggamabar, \aalphabar, w)$ is constructed in that, if one is concerned only with $\spinc$-structures on $X$ which restrict to the one arising from $\xi$, then any holomorphic representative which contributes to the cobordism map must have components which look like the shaded triangle in Figure \ref{figure:local-contact-class} above.
\end{itemize}










 











\newpage
\printbibliography[title = {Bibliography}]

\typeout{get arXiv to do 4 passes: Label(s) may have changed. Rerun}
\end{document}